\documentclass[pdflatex,sn-mathphys-num]{sn-jnl}% Math and Physical Sciences Numbered Reference Style
\usepackage{graphicx}%
\usepackage{multirow}%
\usepackage{amsmath,amssymb,amsfonts}%
\usepackage{amsthm}%
\usepackage{mathrsfs}%
\usepackage[title]{appendix}%
\usepackage{xcolor}%
\usepackage{textcomp}%
\usepackage{manyfoot}%
\usepackage{booktabs}%
\usepackage{algorithm}%
\usepackage{algorithmicx}%
\usepackage{algpseudocode}%
\usepackage{listings}%

\newtheorem{corollary}{Corollary}

\usepackage[utf8]{inputenc}
\usepackage{lmodern}
\usepackage[T1]{fontenc}
\usepackage{graphicx}
\usepackage{subcaption}
\graphicspath{{./}}
\usepackage{geometry}
\usepackage{csquotes}

\usepackage[myheadings]{fullpage}
\usepackage{float}

\newcommand{\ufin}[2]{u_{#1}^{(#2)}} 
\newcommand{\ufint}[2]{u_{#1}^{(#2)}(t)}

\usepackage{color}

\theoremstyle{thmstyleone}%
\newtheorem{theorem}{Theorem}%  meant for continuous numbers
\theoremstyle{thmstyletwo}%

\theoremstyle{thmstylethree}%
\newtheorem{definition}{Definition}%

\begin{document}

\title[Article Title]{Coagulation equations with particle emission}


\author*[1]{\fnm{Joseph} \sur{Klobusicky}}\email{joseph.klobusicky@scranton.edu}

\author[1,2]{\fnm{Matthew} \sur{Rakauskas}}\email{matthew.rakauskas@scranton.edu}

\affil[1]{\orgdiv{Department of Mathematics}, \orgname{University of Scranton}, \orgaddress{\street{800 Linden Street}, \city{Scranton}, \postcode{18510}, \state{PA}, \country{USA}}}

\affil[2]{\orgdiv{Department of Computer Science}, \orgname{University of Scranton}, \orgaddress{\street{800 Linden Street}, \city{Scranton}, \postcode{18510}, \state{PA}, \country{USA}}}


\abstract{  We present a model  for  sticky particles  in which cluster sizes after
a reaction have  $\ell$ fewer total particles than the sum of  their reactants.
 The finite particle system is modeled as a Markov process under a mean-field
assumption for selecting reactants.  The limiting kinetic equations   form
an infinite system of nonlinear differential equations similar to the Smoluchowski
coagulation equations with multiplicative kernel.  We show existence and
uniqueness   for systems whose cluster sizes are either bounded above or
below by the emission size $\ell$. When clusters have at most $\ell$ particles,
well-posedness can be extended until an exhaustion time in which certain
cluster fractions vanish.  For clusters with more than $\ell$ particles,
we prove short-time well-posedness, along with explicit formulas for cluster
sizes and moments.  We also conduct numerical experiments which suggest these
formulas hold until a gelation time, at which an infinite-sized cluster forms.
 }




\maketitle

\section{Introduction}\label{sec1}

We consider a model of    aggregation for sticky particles which involves
the merging  of clusters and the  emission of a fixed positive number of
particles. Denoting  $S_n$ as a cluster of $n\in \mathbb
N_+$ particles and $\ell \in \mathbb N_+$ as the  \textit{emission size},
the reaction of two clusters is    
\begin{equation}
S_i + S_j \rightharpoonup S_{i+j-\ell}, \qquad i+j\ge \ell+1. \label{mainreact}
\end{equation}
The lower bound on the sum of cluster sizes imposes an activation requirement
that each reaction has a
nonempty product. The pure aggregation case of  $\ell = 0$ has been  well-studied,
both for finite-particle stochastic systems \cite{lushnikov1978some,marcus1968stochastic}
and  limiting mean-field coagulation equations \cite{smoluchowski1916drei,mcleod1962infinite,aldous1999deterministic}.
The main object of interest in these studies is the growth of cluster sizes
under various reaction rates.  Closely related are  polymerization processes,
where coagulation reactions vary depending on available reactive sites for
 polymer structures (\cite{ziff1980kinetics,nelson2020kinetic}).  While most
studies focus on reactions with a conservation of total particles,  several
 studies of systems  with mass loss  have been considered  \cite{wattis2004coagulation,singh1996coagulation,banasiak2006coagulation}.
Here, the collision process   (\ref{mainreact}) with $\ell = 0$ is paired
with pure loss
reactions $C_j \rightarrow \emptyset$ for each cluster size $j \ge 1$. 
 
Our motivation for considering a positive emission size in (\ref{mainreact})
comes from a problem in materials science for two-dimensional network microstructure.
Given a two-dimensional soap foam, we may coarsen the foam by rupturing edges.
 If we denote $C_n$ as a cell in the foam with $n$ sides, an edge's rupture
typically affects the side number its four neighboring cells $(C_i,C_j,C_k,C_l)$.
 The three topological reactions of the foam are
\begin{gather}
C_i + C_j \rightharpoonup C_{i+j-4},  \label{facemerge}\\ C_k \rightharpoonup
C_{k-1}, \qquad C_l \rightharpoonup C_{l-1}. \label{edgemerge}
\end{gather} 
In physical and numerical experiments of the rupturing process \cite{klobusicky2024planar,
klobusicky2021markov},  the distribution of cell sizes and areas for aged
foams were found to be highly heterogeneous, consisting  of   a small number
of massive cells interspersed among much smaller cells having few sides.
  Numerical simulations for  mean-field
models following (\ref{facemerge})-(\ref{edgemerge})  suggested the emergence
of a \textit{gel}
as the number of cells becomes infinite \cite{klobusicky2021markov}. In the
context of foams, this corresponds
to the appearance of a massive cell.
Our aim in studying (\ref{mainreact}) is to provide a rigorous analysis of
 mean-field models obeying
a generalization of the second-order reaction  (\ref{facemerge}), allowing
for    a general emission of $\ell \ge 1$ particles after each reaction.

In Section \ref{sec:markov} we construct a finite particle Markovian model
for (\ref{mainreact}) with a state space that tracks frequencies of each
cluster size.  Reactants are  sampled with order and no replacement,  and
selected in proportion to cluster sizes.  Each reaction for the finite-particle
model is a time step, and through scaling with the initial total cluster
number, cluster fractions can be represented as a continuous time cadlag
process. The time scale under this scaling is the extent of reaction, where
time denotes a ratio of total reactions to the initial cluster count. This
use of an `internal clock' is desirable for processes like foam rupture,
where rupture rates are often erratic and bursty \cite{vandewalle2001cascades}.
     In the infinite limit of total clusters, we show formally that fractions
satisfy an infinite system of differential equations which is the main object
of study for the remainder of the paper.  We note that a rigorous proof of
convergence for one and two-species mean-field models with particle loss
is established in
\cite{klobusicky2016concentration,klobusicky2023concentration} through exponential
concentration inequalities.

As opposed to pure aggregation models, positive emission allows for systems
whose cluster sizes remain bounded. In  Section \ref{sec:solns}, we distinguish
between these so-called small-cluster systems and their large-cluster counterparts,
where initial clusters are all larger than the emission size.
We prove basic  properties for both systems and state the main existence
and uniqueness results to be shown in later sections.  Section \ref{sec:reactsmall}
focuses on well-posedness for small clusters until an exhaustion time, at
which   certain cluster fractions approach zero.  This is demonstrated with
a numerical simulation of the Markov process on an example with three species.
 We  also  introduce the concept of reaction classes, which relates small
time growth of clusters to the required number of reactions needed to create
such clusters from initial conditions. Reactions classes are also used in
the well-posedness of large-cluster solutions, presented in Section \ref{sec:largeeu}.
 The equations for large clusters are an infinite system, and so a direct
appeal to well-known existence and uniqueness results  for ODEs is insufficient.
 We build on the  methods used by McLeod \cite{mcleod1962infinite} to show
existence of solutions to Smoluchowski's equations for a generalized class
of multiplicative collision kernels. Uniqueness is shown by constructing
an explicit iterative formula for cluster fractions.

In Sections  \ref{sec:mom}, we  derive a hierarchy of first-order linear
equations for moments of cluster fraction distributions. Explicit formulas
for the second and third moment are derived for monodisperse initial conditions.
We also show that, for $\ell = 1$ and initial conditions of size-two clusters
(dimers), solutions can be written as polynomials of the total mass function.
In Section \ref{sec:sims}, we run simulations of the Markov process for the
large-cluster system.   While solutions are only guaranteed to exist for
small times, we find numerically that cluster fractions and moments agree
until a blow-up of the second moment.  Past this gelation time, we find 
 gelation behavior, in which a single cluster comprises an increasingly large
fraction of total particles.

\section{The $N$-particle Markov model and limiting kinetic equations} \label{sec:markov}

In this section, we construct a  mean-field Markov process to model clusters
evolving under the reaction (\ref{mainreact}). We will work in a  state space
$U = (U_1, U_2, \dots) \in \ell_1(\mathbb
N_+)$ of cluster frequencies, where $U_n$ denotes the number of  $n$-clusters.
Given an initial state $U^{(N)}[0]$ with $N$ clusters, we first describe
a discrete-time Markov process $\{U^{(N)}[t]\}_{t \ge 0}$ , where $t \in
\mathbb N$ denotes the system after $t$ reactions.  
  Of particular import  is our choice of time scale.  Rather than defining
 collision times which are dependent on cluster frequencies,  we instead
impose that
a single reaction occurs at each time step in the discrete model. Thus, our
time scale is with respect to total reactions, or the \textit{extent of reaction}.

The number of particles which are contained  in $n$-clusters is then $nU_n$,
and the
total number of particles in the system is $\sum_{k\ge 1} kU_k$. For 
a single reaction (\ref{mainreact}), we select two clusters at random to
collide.  A natural model will choose clusters in proportion to their size,
so that an  $n$-cluster is selected with
probability 
\begin{equation}
p_{n}[U] = \frac{nU_n }{\sum_{k\ge 1} k U_k }. \label{selectprob}
\end{equation}
To transition from a state $U$ to $U'$, sample two clusters $S_i$ and  $S_j$
from
$U$ with order and  no replacement. Clusters are  selected with cluster selection
probabilities 
 (\ref{selectprob}), conditioned under the  product requirement of
$i+j \ge \ell+1$.
If the clusters selected have sizes $\sigma_1$ and $\sigma_2$, the transitioned
state updates cluster sizes with respect to (\ref{mainreact}) as 
\begin{equation}
U' = (U_1', U_2', \dots), \qquad U_i' =U_i - \mathbf 1_{i =
\sigma_1}- \mathbf 1_{i = \sigma_2}+ \mathbf 1_{i =
\sigma_1+\sigma_2- \ell}. 
 \label{disctrans}
\end{equation}

The discrete-time Markov chain can easily
be extended to a continuous-time cadlag jump
process by scaling time so that one
reaction occurs at each time  $t_k = k/N$ with time increment $\Delta t =
t_{k+1}
- t_k =  1/N$.  This gives a time scale which still represents the extent
of reaction, but is now normalized with respect to the initial  number of
clusters. We can then define normalized cluster fractions as 
\begin{equation}
u_i^{(N)}(t)
= U^{(N)}_i[k]/N, \qquad t \in [t_k, t_{k+1}). \label{cadlag}
\end{equation}

Our main interest  is in an analysis of a formal  law of large
numbers limit $ u_i^{(N)}(t) \rightarrow u_i(t)$ as $N \rightarrow
\infty$, given  convergence of initial conditions $u_i^{(N)}(0) \rightarrow
u_i^0$ where $u^0 = (u^0_1, u^0_2, \dots) \in \ell_1([0, \infty))$.  For
large $N$, effects of sampling without replacement
become negligible
for choosing clusters, so the limiting probability $p_{i,j}[u]$
of selecting two clusters $(S_i,S_j)$ with order and no replacement is given by
\begin{equation}
 p_{i,j}[u] = \frac{iju_i u_j
}{M } \mathbf{1}_{i+j-\ell\ge 1} \qquad \hbox{ with }M = \sum_{m+n-\ell \ge
1} mn u_mu_n. 
 \label{twoprobs}
\end{equation}
Because time is scaled with respect to total collisions, $\dot u_n$ is the
limiting expression of the expected change of $n$-clusters after a single
reaction. This expected value can be decomposed through   birth and death
terms as
\begin{equation}
        \dot u_n = A_n^+[u] - A_n^-[u]. \label{dotuflux}
\end{equation}
 The term $A_n^+[u]$ corresponds to the contribution of the expected value
from reactions which increase the number of $n$-clusters. The reactants for
the birth term take the form $(S_i,S_j),$ where  $i+j
= n+\ell$, $i \neq \ell$, and $j \neq \ell$. In all of these cases, (\ref{mainreact})
increases the number of $n$-clusters by one. Then, noting that $p_{i,j}$
is symmetric, we compute    
\begin{equation}
A_n^+[u] = \sum_{\substack{i+j = n +\ell\\i,j \neq \ell}}p_{i,j}[u]  = \begin{cases}-2p_{n,\ell}[u]+
\sum_{i = 1}^{n+\ell -1}p_{i,n-i+\ell}[u] & n \neq \ell, \ \\
-p_{\ell,\ell}[u]+
\sum_{i = 1}^{2\ell -1}p_{i,2\ell-i}[u] & n =  \ell. \\
\end{cases} \label{aplus}
\end{equation}

  The loss term $A_n^-[u]$  is the contribution of expected value for reactions
in which total $n$-clusters decrease.   Cluster sizes $(S_i,S_j)$ then take
one of the following three forms:
\begin{enumerate}
\item For reactants contained in $\{(S_{n},S_j):j \in \mathbb N_+ \backslash
\{n,\ell\} \} \cup \{(S_{i},S_n):i \in \mathbb N_+ \backslash \{n,\ell\}
\}.$  The reaction (\ref{mainreact}) decreases    $n$-clusters by 1.\item
For $ (S_{n},S_{n})$ with $n\neq \ell$, 
the reaction decreases $n$-clusters by 2.
\item For $ (S_{n},S_{n})$ with $n =  \ell$, the reaction decreases $n$-clusters
by 1.
 
\end{enumerate}
For $n \neq \ell$, we sum probabilities over reactions (1) and (2) to obtain
 \begin{equation}
A_n^-[u] =2p_{n,n}[u]+ \sum_{\substack{j =  (1+ \ell-n)
\vee 1\\j \neq \ell,n}}^{ \infty}2p_{n,j}[u] =  -2p_{n,\ell}[u]+
\sum_{i
=  (1+ \ell-n) \vee 1}^{ \infty}2p_{n,j}[u].  \label{aminus1}
\end{equation}
The maximum in the lower index of the infinite sum in $A^-_n[u]$ enforces
the nonempty product constraint. When $n
= \ell$, a similar calculation considering reactions (1) and (3) gives
\begin{equation}
 A_n^-[u] =-p_{\ell,\ell}[u]+
\sum_{i
=  1}^{ \infty}2p_{\ell,j}[u].  \label{aminus2}
\end{equation}

From (\ref{twoprobs})-(\ref{aminus2}), it follows that for all $n \ge 1$,
 the limiting kinetic
equations for cluster fractions may be written as
\begin{align}
\dot u_n = \sum_{i = 1}^{n+\ell -1} \frac{i(n+\ell-i)}{M}u_i u_{n+\ell
- i} - 2nu_n\sum_{i \ge  (1+ \ell-n) \vee 1} \frac{iu_i}{M}, \label{themainode}
\end{align}
with initial conditions satisfying 
\begin{equation}
u_n(0) = u_n^0\ge 0,  \qquad \sum_{i \ge 1} u_i^0 = 1.  \label{maininits}
\end{equation}

The remainder of this paper is devoted to an analysis of the deterministic
limiting equations (\ref{themainode})-(\ref{maininits}). The pre-limit Markov
process, however, will be simulated at several points to compare against
 numerical and exact solutions of (\ref{themainode})-(\ref{maininits}). 
The implementation for this process is straightforward, using basic sampling
functions.  We found that Python dictionaries serve as a useful data type
for simulations, as they enable us to restrict sampling to nonzero clusters
in an efficient manner. Reducing systems with a million initial clusters
by a decade takes a few seconds on a standard laptop.  We observe that the
variance for cluster fractions is negligibly small across different trials,
and so we plot fractions over a single run when comparing against limiting
quantities.

\section{Solutions} \label{sec:solns}

We now outline solution types of (\ref{themainode}) related to moments and
initial cluster
sizes,  and subsequently state corresponding well-posedness theorems.  The
$k$th moment of a solution is defined as
\begin{equation}
m_k(t) = \sum_{i \ge 1} i^k u_i(t).
\end{equation}
Moments provide summary statistics of cluster distributions.  The zeroth
moment $m_0(t)$ is a total cluster number.  From (\ref{maininits}), clusters
are normalized to the initial cluster count, and so  $m_0(0) = 1$.  The first
moment $m_1(t)$ corresponds to the
total particle number.  For a randomly selected monomer, the first and second
moments  are used in computing  the average size of the cluster containing
the monomer.  In particular, the  average size is the ratio\begin{equation}
\frac{m_2}{m_1} = \sum_{i\ge 1} i\frac{iu_i}{m_1}.
\end{equation}

A finite first moment  is a physically reasonable requirement for a
particle system, and as we shall see, a  finite second moment is crucial
for deriving  expected constant loss rates of total clusters
and particles. 
This motivates us to define the following solution types.

\begin{definition}
For $T>0$, a nonnegative solution $u(t)$ of (\ref{themainode}) is a \textbf{finite
mass solution}   
if $m_{1}(t) = \sum_{n \ge 1} iu_i(t)$ is uniformly convergent for $t \in
[0,T]$. If $m_{2}(t) = \sum_{n \ge 1} i^{2}u_i(t)$ is uniformly convergent
 in $t \in [0,T]$ it is a \textbf{pre-gelation} solution.  
\end{definition}

Since each reaction reduces total clusters by 1 and total particles by $\ell$,
we should expect the following explicit forms for the zeroth and first moments.
\begin{theorem} \label{momentloss}
Suppose $\ell \ge 1$.  If $u(t)$ is a finite mass solution for  $t \in [0,T],$
then
\begin{equation}
 m_0(t) = m_0(0) - t. \label{m0loss}
 \end{equation}If, additionally, $u(t)$ is a pre-gelation
solution, then
\begin{equation}
m_1(t) = m_1(0)-\ell t.  \label{m1loss}
\end{equation}
\end{theorem} 

\begin{proof}
We sum    equations in (\ref{themainode}) over $n \ge 1$ to obtain
\begin{align}
        \sum_{n \ge 1} \dot u_n &=\frac{1}{M}\left[
\sum_{n \ge 1}\left(\sum_{i=1}^{n+\ell-1} i(n-i+\ell)u_iu_{n-i+\ell}
- 2nu_n\sum_{i \ge  (1+ \ell-n) \vee 1} iu_i\right)
\right]\label{suminfinite} \\
&=\frac{1}{M}\left[
\sum_{n \ge 1}\left(\sum_{i=1}^{n+\ell-1} i(n-i+\ell)u_iu_{n-i+\ell}\right)
- 2\sum_{n+i >\ell}  inu_{i}u_{n}
\right] \\
        &:= \frac{ 1}{M}(Q-2M).
\end{align}

Let us show that $M(t)$ is a uniformly convergent series. Writing $M(t) =
\sum_{n\ge 1} a_n$, note that 
\begin{align}
a_n = nu_n\sum_{i \ge  (1+ \ell-n) \vee 1} iu_i \le nu_n \sum_{i \ge 1} iu_i
 = m_1(t)nu_n. 
\end{align}
And so 
\begin{equation}
M(t) \le \sum_{n \ge 1} m_1(t) nu_n,
\end{equation}
which is a uniformly convergent series by assumption and bounded by $m_1(t)^2$.
By the uniform comparison test, it follows that $M$ is a uniformly convergent
series as well.

By changing the order of summation and reindexing, $Q$ can be rewritten
as
\begin{align}
         Q&=  \sum_{n =1}^\infty\sum_{i=1}^{n+\ell-1} i(n-i+\ell)u_iu_{n-i+\ell}
 \\
         &=\left(\sum_{i =1}^{ \ell}\sum_{n=1}^{\infty}+\sum_{i = \ell+1}^{\infty}\sum_{n
=1+i-\ell}^{\infty}\right)
\left[i(n-i+\ell)u_iu_{n-i+\ell}\right] \\
         &=\left(\sum_{i =1}^{ \ell}\sum_{k =1-i+\ell}^{\infty}+\sum_{i =
\ell+1}^{\infty}\sum_{k =1}^{\infty}\right)
\left[ki
u_{i}u_{k}\right] \quad \hbox{ } (\hbox{with }k = n-i+\ell)\\
         &=\sum_{i+k>\ell
}
ki
u_{i}u_{k} = M.
\end{align}

We have shown that $Q$ is  a rearrangement of nonnegative terms
in $M$, so it is also uniformly convergent.  Thus the series $\sum_{n\ge
1} \dot u_{i}(t)$ is  uniformly convergent, and so we may swap the derivative
and
summation operations to obtain $\dot m_0 = \sum_{n \ge 1} \dot u_n = -1$,
establishing   (\ref{m0loss}).

From multiplying $\dot
u_n$ by $n$ and summing over $n\ge 1$ in (\ref{themainode}),
we obtain 
\begin{align}
        \dot  \sum_{n \ge 1} n\dot u_n &=\frac{1}{M}\left[
\sum_{n =1}^\infty\sum_{i=1}^{n+\ell-1} ni(n-i+\ell)u_iu_{n-i+\ell}
- 2\sum_{n+i >\ell}  in^2u_{i}u_{n}
\right] \\
        &:= \frac{ 1}{M}(R_1-R_2).
\end{align}
The uniform comparison test can be applied to show that $R_2$ is uniformly
convergent for pre-gelation solutions.
 The
same reindexing used for the zeroth moment calculations shows that 
\begin{align}
         R_1 &=  \sum_{n =1}^\infty\sum_{i=1}^{n+\ell-1} ni(n-i+\ell)u_iu_{n-i+\ell}
 \\
         &=\left(\sum_{i =1}^{ \ell}\sum_{n=1}^{\infty}+\sum_{i = \ell+1}^{\infty}\sum_{n
=1+i-\ell}^{\infty}\right)
\left[ni(n-i+\ell)u_iu_{n-i+\ell}\right] \\
         &= \sum_{i+k>\ell
}
(k^{2}+ki^{2}-\ell ik)
u_{i}u_{k}\\
 &=2\sum_{i+k>\ell
}
k^{2}i
u_{i}u_{k}-\ell\sum_{i+k>\ell
} ki
u_{i}u_{k} \\
&=  R_2-\ell M.
\end{align}
For a pre-gelation solution, it follows $R_1$ is also uniformly convergent,
and thus  $\dot m_1(t)   = \sum_{n\ge1} n \dot u_n =   -\ell$ establishing
 (\ref{m1loss}).  Note that this
calculation depends on the cancellation of $R_2$ which may be infinite
when  $m_2(t)= \infty$, and thus we  require pre-gelation solutions in order
 for this subtraction to be valid.
\end{proof}

For coagulation equations with no particle emission, $m_1$ remains constant
for all times if $m_2$ is finite.  In some cases there may exist a finite
\textit{gelation time} $ t_{\mathrm{gel}}\ge 0$ in which  $m_2(t) \rightarrow
\infty$ as $t \rightarrow t_{\mathrm{gel}}$. Depending on the model for post-gelation
regimes,  $m_1$ may begin to decrease, which is interpreted as particles
contributing to an infinite-sized \textit{gel}.  In our case, gel formation
would occur when $m_2(t_{\mathrm{gel}}) = \infty$, and  we should expect
that $m_1(t) <m_1(0)- \ell t$ in a post-gelation regime.  Numerical evidence
of this phenomenon is presented in Section \ref{sec:sims}.

One major departure from previously studied coagulation equations is that
particle emission allows for solutions with bounded cluster sizes.  Indeed,
if reactants have size at most $\ell$, then the product size will also be
bounded by $\ell$.  
\begin{definition}
We refer to a cluster $S_n$  as \textbf{small} if  $n \le
\ell$, and   \textbf{large} if $n>\ell$. Similarly, we call the non-negative
initial conditions in (\ref{maininits}) \textbf{small} if $u_i^0 = 0$ for
$i > \ell$, and $\textbf{large}$ if   $u_i^0 = 0$ for $i \le \ell$.
\end{definition}
A key observation is that (\ref{mainreact}) can neither produce large clusters
from small clusters, nor small clusters from large clusters. This would suggest
that solutions with small  initial conditions have no ``spontaneous generation'',
meaning $u_i(t) \equiv 0$ for $i>\ell$.  Similarly, we should expect $u_i(t)
\equiv 0$ for $i \le \ell$ for large initial conditions.    For finite mass
solutions, it can be shown that the total number of both
small clusters $m_{0;\mathrm{s}}
= \sum_{i \le \ell} u_i$ and big clusters $m_{0;\mathrm{b}} = \sum_{i
> \ell} u_i$ are nonincreasing in time.  We also note that the total number
of particles for
 big clusters $m_{1;\mathrm{b}} = \sum_{i
> \ell} iu_i$ is also nonincreasing, although this is not necessarily true
for
small clusters  $m_{1;\mathrm{s}} = \sum_{i
\le \ell} iu_i$.  For instance, when $\ell = 3$ the reaction $S_1+S_4 \rightharpoonup
S_2$ increases the total particles in small clusters by 1.

\begin{theorem} \label{bigsmallgrow} (Growth of small and large clusters).
Let $\ell \ge 1. $  For a finite mass solution $u(t)$ on $t \in [0,T]$, 
\begin{align*}\dot m_{0;\mathrm{s}} &= -\frac{1}{M} \left(\sum_{k \le
\ell} ku_k\right )^2-\frac{2}{M}
\left(\sum_{k
\ge 2\ell}
\sum_{j = 1}^{\ell} +\sum_{k = \ell+1}^{2\ell-1}\sum_{j
= 2\ell-k+1}^{\ell}\right)[kju_ku_j] \le 0, \\
\dot m_{0;\mathrm{b}} &= -\frac{1}{M} \left(\sum_{k
>
\ell} ku_k\right )^2- \frac{2}M\sum_{k=
\ell+1}^{2\ell-1} \sum_{j = 1}^{2\ell-k} kju_ku_j \le 0. \\
\end{align*}

Additionally, if $u(t)$ is  a pre-gelation solution,
\begin{align*}
\dot m_{1;\mathrm{s}} &=  -\frac{\ell}M\left(\sum_{k
\le \ell} ku_k\right )^2-
\frac{2}M\sum_{k = \ell+1}^{2\ell-1}\left( \sum_{j = 1}^{2\ell-k} (\ell-k)
kju_ku_j+\sum_{j
= 2\ell-k+1}^{\ell} kj^{2}u_ku_j\right)\\&\qquad-\frac{2}{M}\sum_{k
\ge 2\ell}
\sum_{j = 1}^{\ell} kj^{2}u_ku_j, \\
\dot m_{1;\mathrm{b}} &=  -\frac{\ell}M\left(\sum_{k
> \ell} ku_k\right )^2-
\frac{2}M\sum_{k = \ell+1}^{2\ell-1}\left( \sum_{j = 1}^{2\ell-k} k^2ju_ku_j+\sum_{j
= 2\ell-k+1}^{\ell}(\ell-j) kju_ku_j\right)\\&\qquad -\frac{2}{M}\sum_{k
\ge 2\ell}
\sum_{j = 1}^{\ell} (\ell-j)kju_ku_j \le 0.  
\end{align*}
 
\end{theorem}

\begin{proof}

Formally, the rates of growth for the zeroth and first  moments restricted
to large and small clusters (denoted with subscripts of ``b'' and ``s'',
respectively) can be found by computing the expected
 change of total clusters $\Delta C_\mathrm b, \Delta C_\mathrm s $ and particles
$\Delta P_\mathrm b, \Delta P_\mathrm s$ after a single
reaction.  These changes can be broken into five cases when
considering the size of reactants $j$ and $k$.  See Table \ref{bigsmalltable}
for the values
of $\Delta C$ and $\Delta P$ under each of these
cases. Theorem \ref{bigsmallgrow} is then shown by summing probabilities
of each case occurring multiplied by the contribution of cluster sizes for
each case.   Rigorously, these expressions can be found by summing $\{\dot
u_i\}_{i>\ell}$
and $\{i \dot u_i\}_{i >\ell}$ in $(\ref{themainode})$ and reindexing, similar
to the proof in Theorem (\ref{momentloss}).  As with establishing (\ref{m1loss}),
the  sums involved for the total particle count require that $m_2(t)<\infty$
in order to be valid.  \end{proof}

\begin{table}
\begin{tabular}{|c|c|c|c|c|}\hline
\textbf{Reactant sizes} & $\Delta C_\mathrm s$ & $\Delta P_\mathrm s$ & $\Delta
C_\mathrm{b}$ & $\Delta P_\mathrm{b}$ \\\hline
$\ell+1  \le  k \le 2\ell-1$ and  $1\le j \le  2\ell-k$ & $0$ & $k-\ell$
& $-1$ & $-k$ \\\hline
$\ell+1  \le  k \le 2\ell-1$ and  $2\ell-k+1\le j\le \ell$ & $-1$ & $-j$
& 0 & $j-\ell$ \\\hline
$k \ge 2\ell$, and $1\le j \le \ell$ & $-1$ & $-j$ & 0 & $j-\ell$ \\\hline
$1 \le k,j \le \ell$ & $-1$ & $-\ell$ & 0 & 0 \\\hline
$k,j \ge \ell+1$ & 0 & 0 & $-1$ & $-\ell$ \\\hline
\end{tabular}
\caption{Growth of cluster sizes with respect to reactant sizes $j$ and $k
 $. For the first three cases, we assume $j<k$. These indices may be swapped
to produce the same changes in total clusters and particles. This is reflected
in the factors of 2 in Theorem \ref{bigsmallgrow}.}
\label{bigsmalltable}
\end{table}

\begin{theorem} \label{nosg} (No spontaneous generation).
For a finite mass solution $u(t)$ on  $t \in [0,T]$,

(i) if $u$ has small initial conditions, then $u_i(t) \equiv 0$ for $i>\ell$.

(ii) if $u$ has large initial conditions, then $u_i(t) \equiv 0$ for $i\le
\ell$.
  
\end{theorem}

\begin{proof}
(i) For small initial conditions, $m_{0;\mathrm b}(0) = 0$ and $\dot m_{0;\mathrm
b}(0) \le  0$ from  Theorem \ref{bigsmallgrow}. Since solutions are nonnegative,
this implies $u_i(t) \equiv 0$ for $i>\ell$. The proof for part (ii) is similar.
\end{proof}

Theorem \ref{nosg} shows that small or large initial conditions remain so
for positive times, and so for these cases we can restrict our well-posedness
analysis to a smaller system of equations.  For small initial conditions,
(\ref{themainode})
becomes finite.  For $1 \le n \le \ell$,  
\begin{align}
\dot u_n = \sum_{i = n}^{\ell} \frac{i(n+\ell-i)}{M}u_i u_{n+\ell
- i} - 2nu_n\sum_{i =  1+ \ell-n  }^{\ell} \frac{iu_i}{M},  \label{mainsmall}
\end{align}
with 
\begin{equation}
M = \sum_{i = 1}^\ell \sum_{j = 1+\ell-i}^\ell iju_iu_j.
\end{equation}

For large initial conditions, there are no restrictions on which clusters
can
interact, since products always have at least $\ell+2$ particles. The total
probability then simplifies as 
\begin{equation}
M = \sum_{i,j>\ell} iju_iu_j =m_1^2. \qquad 
\end{equation}
 For $n> \ell$, the large-cluster equations are
\begin{align}
\dot u_n &= \sum_{i =
\ell+1}^{n-1} \frac{i(n+\ell-i)}{m_1(t)^2}u_i u_{n+\ell
- i} - \frac{2nu_n}{m_1(t)}, \qquad  \label{realmaineqn}
\quad  
\end{align}  
where the sum in (\ref{realmaineqn}) is understood to vanish when $n = \ell+1$.
 For both (\ref{mainsmall}) and (\ref{realmaineqn}), we keep initial condition
requirements (\ref{maininits}).

Under a time reparametrization, (\ref{realmaineqn})
takes a form of coagulation equations with multiplicative reaction rates.
 This reparametrization is given by  ``$q$-coordinates'' $q_k(t) = u_k(t)/m_1(t)$.
 In the pre-gelation regime, $\dot m_{1} = -\ell$  and
(\ref{realmaineqn}) becomes
\begin{equation}
m_1 \dot q_n  
 = \sum_{i = \ell+1}^{n-1} i(n+\ell-i)q_i q_{n+\ell - i}- (2n-\ell)q_n. 
\label{qness}
\end{equation}
Under this rescaling, total particle numbers $\{kq_k(t)\}_{k \ge 1}$ are
a probability measure at each time $t$, as opposed to the original total
particle fractions $\{ku_k(t)\}_{k\ge 1}$ which  are scaled with respect
to  initial total particles and are  a sub-probability measure. In the  case
when $\ell = 0$, total particles are conserved, meaning $m_1 \equiv 1 $ and
(\ref{realmaineqn}) is equal to the  Smoluchowski coagulation equations with
multiplicative reaction rates of $2nk$ between $k$ and $n$ clusters.  Since
the change of variables is a time reparametrization, solutions   $\{q_k(t)\}_{k\ge
1}$ and  $\{u_k(t)\}_{k\ge 1}$ have equivalent trajectories \cite[Section
1.4.1]{chicone2006ordinary}.

\begin{definition}
For  $T>0$,  we call a nonnegative solution $u_\mathrm{s}(t) = (u_1(t), \dots,
u_{\ell}(t))
$ of (\ref{mainsmall}) and (\ref{maininits}) a \textbf{small-cluster solution}
if $u_{k}(t) \in C^1([0,T], [0, 1])$ for $k \le \ell$.
.  We call a nonnegative solution
$u_b(t) = (u_{\ell+1}(t), u_{\ell+2}(t), \dots ) $  of (\ref{realmaineqn})
and (\ref{maininits})
  a \textbf{large-cluster solution} if
$u_{k}(t) \in C^1([0,T], [0, 1])$ for $k \ge \ell+1$.
\end{definition}

Our well-posedness theorem for  small clusters guarantees a unique solution
up to an exhaustion time
\begin{equation}
t_{\mathrm{ex}}
:= \inf \{s\ge 0: M(s^-)= 0\},
\end{equation}
 when some (but not necessarily all) clusters
vanish. 

\begin{theorem} \label{smalleuthm}
For small initial conditions, the system  (\ref{mainsmall})
has a unique small-cluster solution on $[0,t_{\mathrm{ex}})$.  This solution
is analytic.
\end{theorem}
Since the system is finite, the second moment also remains finite, and no
gelation occurs. Thus, any small-cluster solution is a pre-gelation
solution. 
The proof of Theorem \ref{smalleuthm} is presented in Section \ref{sec:reactsmall}.
 The main hurdle is in showing that solutions remain nonnegative, which is
addressed
with the concept of reaction classes that provide small time asymptotics
of cluster fraction growth. 

For large clusters, we will show local existence and uniqueness for pre-gelation
solutions.
\begin{theorem} \label{largeeuthm}
For large initial conditions which are bounded (i.e., there exists $L\ge
\ell+1$
such that $u_i^0 = 0$ if $i> L$), there 
is a positive time $T>0$  such that there is a unique large-cluster solution
of (\ref{realmaineqn})
which is also  pre-gelation solution $u(t)$.
Cluster fractions for this unique solution are smooth (i.e. $u_i \in C^\infty([0,T],
[0,1])$ for $i \ge \ell+1$). \end{theorem}

Because (\ref{realmaineqn}) is an infinite set of differential equations,
we cannot directly appeal  to well-known existence and uniqueness  theory
 for
finite systems of ODEs.   In Section \ref{largeeuthm}, we give a proof of
Theorem \ref{largeeuthm}
which applies the Arzela-Ascoli theorem and a diagonalization argument  to
a sequence of truncated solutions which restrict cluster sizes.  It is similar
to McLeod's  proof found  \cite{mcleod1962infinite} for
establishing existence and uniqueness for the Smoluchowski equations with
multiplicative kernels and monodisperse initial conditions ($u_1 = 1$ and
$u_j = 0$ for $j\neq 1$),  but some extra care is needed to account for 
 positive emission and the broader class of initial conditions with finite
support. Additionally,  \cite{mcleod1962infinite} only establishes existence
of solutions, whereas we show uniqueness directly by constructing an iterative
solution formula which relies on pre-gelation conditions.   
\section{Reaction classes and small-cluster solutions} \label{sec:reactsmall}

Existence and uniqueness
for (\ref{mainsmall}) follow quickly from standard theory:

\begin{theorem} \label{smallgeneu}
  There is a unique analytic solution
  $u(t) \in C^{\omega}([0,t_{\mathrm{ex}}))$ of (\ref{mainsmall}) with initial
conditions (\ref{maininits}).
\end{theorem}
  
\begin{proof}
We can write (\ref{mainsmall})
as $\dot u = F(u)$,
where $F$ is analytic at $M \neq 0$. From the Cauchy-Kovalevskaya theorem,
we are guaranteed a unique analytic solution $u(t)$  in an interval $[0,t^*)$
  for some $t^*>0$. From standard arguments,
this  can interval can be extended to $[0,t_{\mathrm{ex}})$. Specifically,
suppose for the sake of contradiction that we can extend
solutions  to a maximum interval
$(0,T)$ for which $\lim_{t\rightarrow T^-} M(t)>0$ and $T<t_{\mathrm{ex}}$.
  Simple estimates show that $|\dot u_i|<3$ for $1 \le i \le \ell$, and so
solutions must be bounded in $[0,T)$. Thus for each
$u_i(t)$ there will be at least one accumulation point at time $T$.  In fact,
there is only one, as the existence of more than one such accumulation point
implies unbounded derivatives.  Since this limit exists,
we can further extend our interval of existence, in contradiction to the
maximality of $T$. 

\end{proof}

It remains to show that the solution is nonnegative. Toward this end we will
need notation
which tracks how many reactions are required to create clusters of a certain
size.  
\begin{definition}
Given initial conditions $\{u_n^0\}_{n \ge 1}$, the \textbf{$k$th reaction
class} for is defined recursively as
\begin{align}
\mathcal S_0 &= \{n:u_n(0)>0\}, \\
\mathcal S_k &= \{n \in \mathbb N_+: n \notin \cup_{i <k} \mathcal S_i \hbox{
and }  \exists n_1, n_2 \in
\cup_{i <k} \mathcal S_i \hbox{ such that } n_1+n_2-\ell = n \} \quad \hbox{
for } k \ge 1.
\end{align}
The set of all \textbf{attainable cluster sizes} is denoted $\mathcal S_\infty
= \cup_{i\ge
0} \mathcal S_i$.
We say  $s_n\ge 0$ is the \textbf{reaction number} for $n$-clusters if $n
\in \mathcal S_{s_n}$.
\end{definition}

 Thus, the $k$th reaction class $\mathcal S_{k}$ are those cluster sizes
 which require $k$ reactions to form, and the reaction number $s_n$ is
the required number of collisions to form an $n$-cluster. Under generic initial
conditions, we should not expect a sequence $\{a_n\}_{n\ge 0}$ to be increasing
if $a_n \in \mathcal S_n$. Finding $\mathcal S_k$ is a ``taxed'' version
of the famous change making problem: given coins that only take certain values
and a desired value $A$, determine the least number of coins whose values
sum to $A$. Determining
$\mathcal S_\infty$ is the ``taxed'' version of the Frobenius coins problem:
for coins taking certain values, determine all values that  can be attained
from adding values of   multiple coins    \cite{alfonsin2005diophantine}. Solving this problem is known to be NP hard \cite{cai2009canonical}.

We now show a result that gives a relation between reaction classes and
small time growth of clusters.  Establishing this result
is somewhat more straightforward for large-cluster solutions, so we will
give this proof
first and then show how it can be  modified to work for small-cluster solutions.

\begin{theorem} \label{largecollproof}
Let $u(t)$ be a large-cluster solution. and let $j\ge 0$ be given. For all
 $n\ge 1$,  

(i) If  $s_n>j$, then   $\frac{d^j
}{dt^j}u_n(0)=
0$.

(ii)If $s_{n} = j$, then  $\frac{d^j }{dt^j} u_n(0)
>0$.
\end{theorem}

\begin{proof}

We will work in $q$-coordinates (\ref{qness}), rewritten as
\begin{equation}
m_1 \dot q_n  
 + (2n-\ell)q_n= \sum_{i = \ell+1}^{n-1} i(n+\ell-i)q_i q_{n+\ell - i}. 
\label{qness2}
\end{equation}
Since $m_1(0)$ is positive, it is straightforward to show that if (i) and
(ii) hold for $\{q_n\}_{n\ge 0}$, then they both hold for $\{u_n\}_{n\ge
0}$ as well.
We show (i) through a strong induction argument on $j$. Certainly, (i) holds
for  $j = 0$, since $q_n(0) = 0$ for $n \notin
\mathcal S_0$.  We now assume (i)  holds for $0 \le i \le j-1$.
 Suppose $n \in \mathcal S_k$, with $k>j$.  We take $j-1\le k-2$ derivatives
of (\ref{qness2}). On\ the left hand side,
from the inductive hypothesis,\begin{equation}
        \frac{d^{j-1} }{dt^{j-1}}\Big |_{t = 0}(m_{1} \dot q_n + (2n-\ell)q_n
) =
m_{1}q^{(j)}_n(0).
 \end{equation}
On the right hand side (\ref{qness2}), from the Leibniz rule, taking $j-1$
derivatives results
in terms of the form $q_i^{(j_1)}q_{n+\ell-i}^{(j_2)}$ with $j_1+j_2  = j-1
\le k-2$. We also note the $s_i+s_{n+\ell-i} \ge k-1$, since otherwise an
$i$ cluster and $n+\ell-i$ cluster could react and create an $n$ cluster
with fewer than $k$ total reactions, a contradiction to $n \in \mathcal
S_k$.
Thus either $j_1< s_i$ or $j_2<s_{n+\ell-i}$, and by the inductive hypothesis
one of the terms in $q_i^{(j_1)}q_{n+\ell-i}^{(j_2)}$ must vanish at 0. 
This implies $q^{(j)}_n(0) = 0$, establishing (i) for $j$ and therefore establishing
the inductive
hypothesis to prove (i).

To show (ii) holds, we use (i) and an induction argument on $j\ge 0$. For
$j= 0$, (ii) holds from the definition of $\mathcal S_0$.  For
the inductive step, assume (ii) holds for  $0 \le i \le j-1$. Suppose $n
\in \mathcal S_j$. Taking
$j-1$ derivatives of (\ref{qness2}), we find that the left-hand side reduces
to $m_{1}q^{(j)}_n(0)$.  The indices in the sum of the right-hand side  can
be partitioned
as $I = \{\ell+1, \dots, n-1\} = A_1 \sqcup A_2$, with
\begin{equation}
  A_1 = \{i \in I:s_i
+ s_{n-i+\ell} = j-1\}, \qquad   A_2 = \{i \in I:s_i
+ s_{n-i+\ell} > j-1\}. 
\end{equation}
 By
the same methods used to show  (i),   all terms corresponding
to indices in $A_2$
will
vanish.  For $i \in A_1$, by the induction hypothesis on (ii), the only nonvanishing
terms evaluated at 0 are those of
the form
\begin{equation}
i(n+\ell-i)q_i^{(s_i)}(0)q_{n+\ell - i}^{(s_{n+\ell-i})}(0)>0,
\end{equation}  
and therefore $q^{(j)}_n(0)>0$, establishing (ii).

\end{proof}

The proof for the small-cluster case is similar, albeit more messy since
$q$-coordinates do not have the same clean representation as (\ref{qness})
found for large clusters.

\begin{theorem} \label{smallcollproof}
For a small-cluster solution $u(t)$, statements (i) and (ii) in Theorem \ref{largecollproof}
also hold. 
\end{theorem}
\textit{Proof sketch.} The proof mirrors the induction arguments in Theorem
\ref{largecollproof}. For the induction step part (i), we again take $j-1<k-2$
derivatives of (\ref{mainsmall}).  The first   sum has terms proportional
to $(\frac 1{M(0)})^{(j_1)} u_i(0)^{(j_2)} u_{n-i+\ell}(0)^{(j_3)}$ with
$j_1+j_2+j_3
= j-1$ and $j_1,j_2,j_3\ge0$.  Then $j_2+j_3<k-2$, and so each of these terms
must vanish.  It is clear that all terms in
the second sum must also  vanish.
For the inductive step in (ii), after taking $j-1$ derivatives, the only
terms
which do not vanish will be those with form 
\begin{equation}
  \frac{i(n+\ell-i)}{M(0)}u_i^{(s_i)}(0)u_{n+\ell - i}^{(s_{n+\ell-i})}(0)>0.
  \end{equation}
\qed

 \begin{corollary}\label{poscor}(Short-time asymptotics) For solutions $u(t)$ with
small or large
initial conditions, if  $n \in  \mathcal S_k$, then there
is a positive constant $C_n>0$ such that  $u_n(t) \sim C_nt^k$ as $t \rightarrow
0^+$.
\end{corollary}
\begin{proof}
This is a direct consequence  of Taylor's theorem and Theorems \ref{largecollproof}
and \ref{smallcollproof}.
\end{proof}

We can now establish nonnegativity for small-cluster solutions. Theorem \ref{smalleuthm}
 subsequently follows from Theorem \ref{smallgeneu}.

\begin{theorem} \label{smallpos}
If $u(t)$ is the unique small-cluster solution to (\ref{mainsmall}) then
$u(t)\ge 0$ for $t
\in [0, t_{\mathrm{ex}})$.
\end{theorem}

\begin{proof}
We will show nonnegativity for some interval about 0, noting that it is straightforward
to show that the interval
of nonnegativity  may
be extended  to $[0,t_{\mathrm{ex}})$. From Corollary \ref{poscor}, $u_n(t)>0$
if
$n \in \mathcal S_\infty$ for some small time interval.  If $i \notin \mathcal
S_\infty$, then $u_i^{(k)}(0)
= 0$ for all $k \ge 0$. As  $u_i(t)$ is analytic, it must therefore be identically
zero for a small time interval.    \end{proof}

\subsection{An example with three species}

 We now consider a simple   example of clusters with size of at most three.
We set an emission size of $\ell = 3$, and initial conditions of $u_1(0)
= p$, $u_2(0) = q$
and $u_3(0) = r$ with $p+q+r = 1$. This gives
a closed system, limited to monomers $S_1$, dimers $S_2$, and trimers $S_3$.
There are   four permissible reactions:\begin{align}
S_2 + S_2 \rightharpoonup S_{1}, \qquad  S_1 + S_3 \rightharpoonup S_{1},
\qquad  S_3 + S_2 \rightharpoonup S_{2}, \qquad
 S_3 + S_3\rightharpoonup S_{3}.
 \end{align} 
 
\begin{figure}[htbp]
    \centering
    \begin{subfigure}[b]{0.45\textwidth}
        \includegraphics[width=\textwidth]{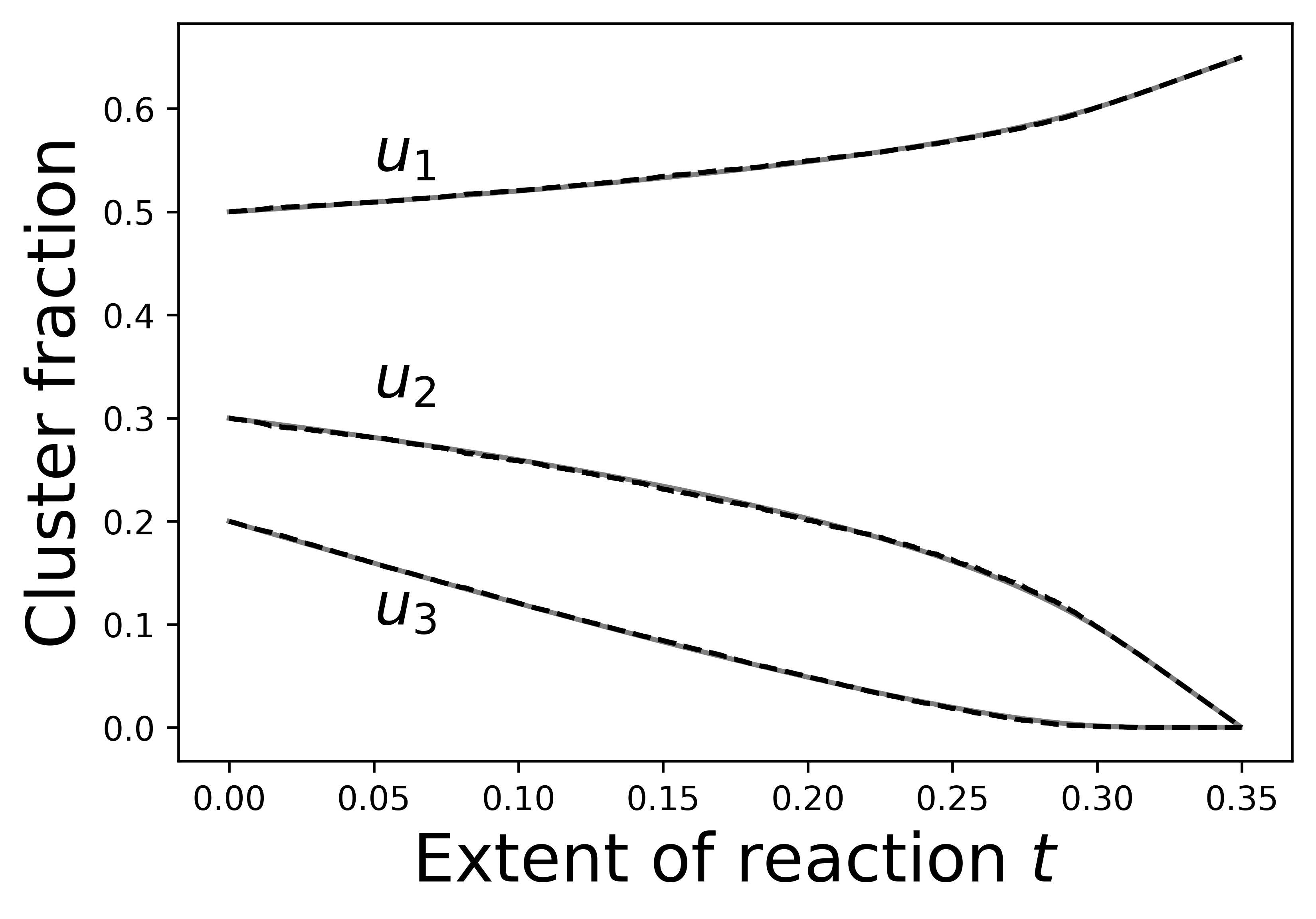}
    \end{subfigure}
    \hfill
    \begin{subfigure}[b]{0.45\textwidth}
        \includegraphics[width=\textwidth]{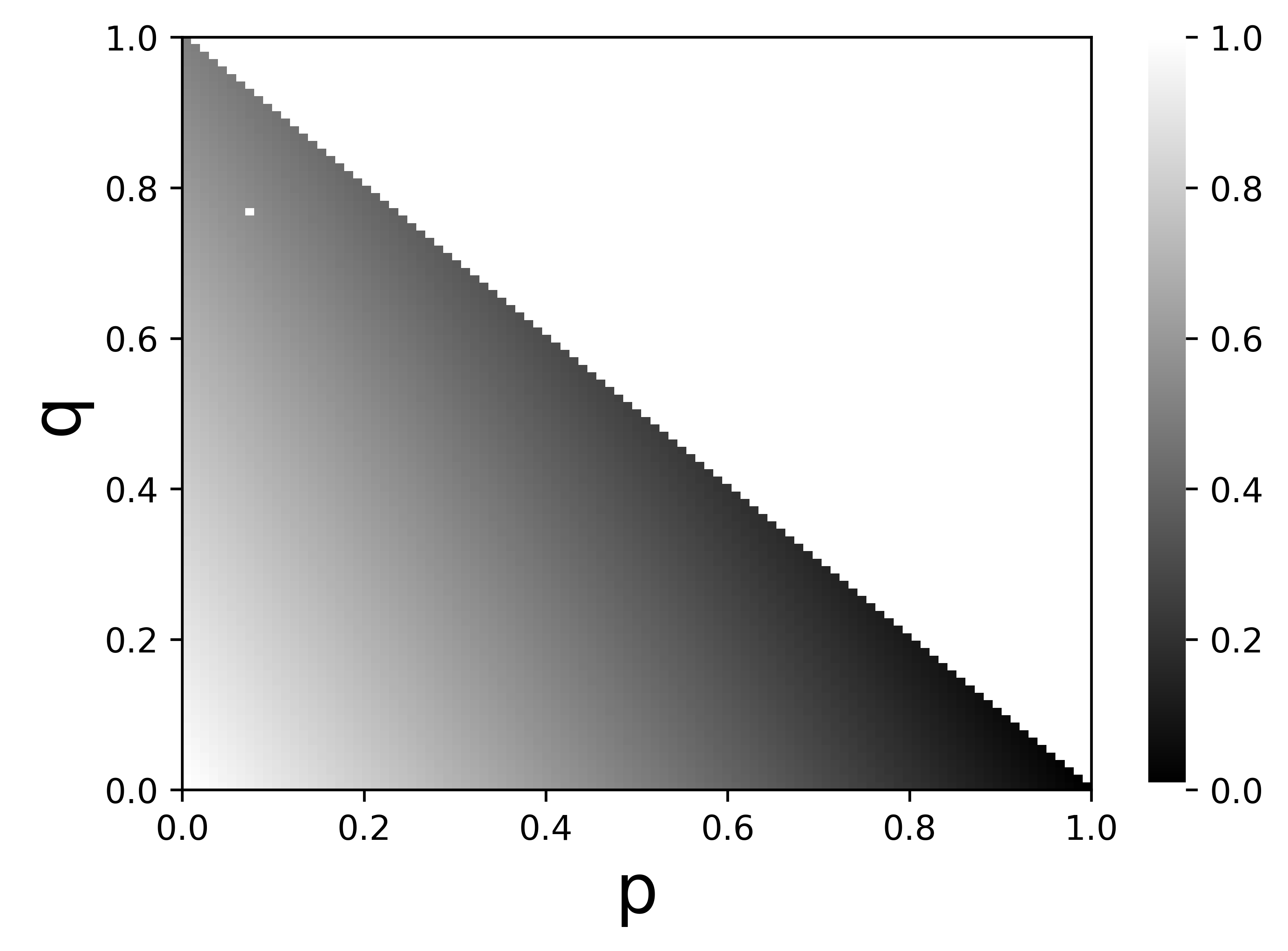}
    \end{subfigure}
    \caption{\textbf{Simulations of the three species system (\ref{threespec}).
\textbf{Left}}: Numerical solutions (transparent solid lines) and the Markov
process (\ref{cadlag}) with $10^5$ initial clusters (dashed lines) of cluster
fractions
for initial conditions $u_1(0) = 0.5, u_2(0) = 0.3, u_3(0) = 0.2$. \textbf{Right}:
A heat map of exhaustion times for (\ref{threespec}) with initial fractions
$u_1(0) = p$ and $u_2(0) = q$ with $p+q \le 1$.\label{threespecpic} }
\end{figure}

  From (\ref{mainsmall}), the closed   system is 
\begin{align}
\dot u_1 &= \frac{4u_2^2}{M}, \qquad
\dot u_2 = \frac{-8u_2^2}{M}, \qquad
\dot u_3 = \frac{-6u_1u_3-12u_2u_3-9u_3^2}{M} ,\label{threespec}
\end{align}
with
\begin{equation}
M = 4u_2^2+6u_1u_3+12u_2u_3+9u_3^2. 
\end{equation}

When $r = 0$,  only dimers can react, and so we can easily verify that $u_1(t)
=  p+t,$ $u_2(t) = 1-p-2t$, and $u_3 \equiv 0$ is a solution for $t \in [0,(1-p)/2)$.
When $q = 0$, monomers are conserved, and we obtain $u_1(t) \equiv p, u_2(t)
\equiv 0,  u_3(t) = r-t$ for $t \in [0, 1-p)$. At these exhaustion times,
 $u_2,u_3$, and subsequently $M$ vanish.

We can use moments to help simplify the system when both $q$ and $r$ are
positive.  The zeroth and first moments are
\begin{align} 
m_0(t) = \sum_{i = 1}^3 u_i(t) =  1-t, \qquad  
m_1(t) = \sum_{i = 1}^3 iu_i(t) = 3-2p-q-3t. \label{threecon}
\end{align}
From these two relations, we can express $M$
as a function of  only $u_2$ and $t$ by  observing
\begin{equation}
M = (3-2p-q-3t)^{2}-u_1^2-4u_1u_2, \qquad u_1= (q+2p-u_2)/2.
\end{equation}
We can then solve for $u_2$ numerically, and directly obtain $u_1$ and $u_3$
from (\ref{threecon}).  Since $u_1$ is nondecreasing, there will always be
a positive
fraction of monomers at all positive times. However, we now show that for
 $p,q \in (0,1)$, dimers and trimers
will  simultaneously exhaust at $t_{\mathrm{ex}}$. \begin{theorem} (Simultaneous
exhaustion of dimers and trimers)
For $p,q \in (0,1)$ both $u_2(t)$ and $u_3(t)$ are positive for 
$t\in [0, t_{\mathrm{ex}})$.  \end{theorem}

\begin{proof}
 For the sake of contradiction, suppose there is a first time $s^*< t_{\mathrm{ex}}$
for which $u_2(s^*) = 0$ (the proof for assuming $u_3$ exhausts first is
similar). Let  $u(t)$
  denote the unique solution of (\ref{threespec})  for the time interval
$[0,t_\mathrm{ex})$.
  We can construct a second distinct solution $\tilde u(t)$ with initial
conditions
\begin{equation}
\tilde u(s^*) = (u_1(s^*),0,u_3(s^*)).
\end{equation}
In an interval centered at 0, this has a solution of $\tilde u_2 \equiv 0$,
$\tilde u_1(t) \equiv s^* $ and $\tilde u_3(t) = u_3(s^*)-t$. $
$This contradicts uniqueness of solutions. Thus both $u_2$ and $u_3$ must
remain positive in $[0,t_{\mathrm{ex}})$.   
\end{proof}

In Fig. \ref{threespecpic} we show agreement of solutions of (\ref{threespec})
with a trajectory of $10^5$ initial particles and initial conditions of $(p,q,r)
= (.5, .3, .2)$.  We also show a heat map of exhaustion times over all possible
initial conditions. We note that for certain values of $(p,q)$,   (\ref{threespec})
is stiff, and an implicit Runge Kutta method (Radau IIA family of order 5)
\cite{hairer1991ii} is used to ensure stability when computing numerical
solutions.

\section{Existence and uniqueness of large-cluster solutions} \label{sec:largeeu}

In this section we provide a proof for Theorem  \ref{largeeuthm}.  This is
done in four parts.  The first part constructs a truncated version of  (\ref{realmaineqn})
which restrict products of  (\ref{mainreact}) to have at most $N$ particles.
 These systems maintain the constant loss rates of total cluster particle
given in (\ref{m0loss})-(\ref{m1loss}).  In the second part, we find uniform
monomial growth bounds for cluster fractions.  In the third part, we use
the uniform bounds and a diagonalization argument similar to the proof of
the Helly's selection theorem to construct a convergent subsequence of cluster
fractions whose limit solves (\ref{realmaineqn}).  Finally, in the fourth
part we show that all solutions satisfying the convergence criteria can be
written as an iterative sequence of integral equations, proving uniqueness.

(\textit{I: truncated equations}). Suppose we are interested in a solution
of $(\ref{realmaineqn}) $ with  initial conditions of $u^0 = (u_{\ell+1}^0,
u_{\ell+2}^0,
\dots)$ such that $\sum_{i \ge \ell+1 }u_i^{0} = 1$ and $u_i^0 = 0$ for $i
>L$. For each $N\ge 2L + \ell$, we consider the following set of truncated
equations:
\begin{align}
\dot u_n^{(N)} &=\frac{1}{M^{(N)}}\left[ \sum_{k = \ell+1}^{n-1} k(n-k+\ell)
\ufin{k}{N} \ufin{n-k+\ell}{N} - 2n\ufin{n}{N}\sum_{k = \ell+1}^{N-n+\ell}
 ku_k \right], \quad n = \ell+2, \dots, N-1, \label{ufineqell}\\
\dot u_{\ell+1}^{(N)} &= -\frac{2(\ell+1)u_{\ell+1}}{M^{(N)}}\sum_{k = \ell+1}^{N-1}ku_k^{(N)},\\
\dot u_N^{(N)} &= \frac{1}{M^{(N)}}\sum_{k = \ell+1}^{N-1} k(N-k+\ell) \ufin{k}{N}
\ufin{n-k+\ell}{N}. \\
u^{(N)}_n(0) &= u_n^{0}, \qquad   \ell+1 \le n \le L, \\ u^{(N)}_n(0) &=
0,  \qquad n>L,
\label{unfinspec}
\end{align}
where the total probability corresponding to permissible reactions is 
\begin{equation}
        M^{(N)} = \sum_{j = \ell+1}^{N-1}\sum_{k = \ell+1}^{N-j+\ell} jk\ufin{j}{N}
 \ufin{k}{N}.
\end{equation}
For $i\ge 0$, the $i$th moments is given by 
\begin{equation}
m^{(N)}_i(t) = \sum_{n = 1}^N n^i \ufint{n}{N}. 
\end{equation}

The truncated system has many of the same properties as the infinite system.

\begin{theorem}
Let $T_1= \ell/(2(\ell+1))$. For $t \in  [0,T_{1}],$  
\begin{enumerate}        
        \item The truncated system (\ref{ufineqell}) has a unique solution
$u^{(N)}(t)$ for $t \in [0,T_{1}]$.        \item $u_i^{(N)}(t)>0$
for $i \in \mathcal S_\infty$ and $u_i^{(N)}(t) = 0$  for $i \notin \mathcal
S_\infty$. 
\item  The truncated equations have the same loss rates for the zeroth
and first moments: 
\begin{equation}
m_0^{(N)}(t) = 1-t, \qquad m_1^{(N)}(t) = m_1(0)-\ell t. \label{finmom}
\end{equation}
\item $M^{(N)}(t)\ge 1.$ 
\end{enumerate}
\end{theorem}

\begin{proof}
 The proofs for parts (1) and (2) are largely similar to those establishing
 the well-posedness and nonnegativity for small-cluster systems, and so we
omit them.
The moment equations (\ref{finmom}) can be shown by the  reindexing argument
in Theorem \ref{momentloss}.
To show (4), we note that the smallest cluster has $\ell+1$ particles.  At
$t = 0$, all clusters can interact since $N>2L$ , and so 
\begin{equation}
M^{(N)}(0) = m_1(0)^2\ge (\ell+1)^2.
\end{equation}
It can be directly verified that  
\begin{equation}
\sum_{i
= \ell+1}^L\dot u^{(N)}_i(t)\ge  
-2 \qquad   \Rightarrow  \qquad \sum_{i
= \ell+1}^L u^{(N)}_i(t) \ge 1-2t.
\end{equation}
   It then follows that 
\begin{align}
M^{(N)}(t) &\ge \sum_{i = \ell+1}^L iju_i^{(N)}(t) u_j^{(N)}(t) = \left(\sum_{i
= \ell+1}^L
iu_i^{(N)}(t)\right)^2  \ge (\ell+1)^2\left(\sum_{i = \ell+1}^L
u_i^{(N)}(t)\right)^2  \\
&\ge (\ell+1)^2\left(1-2t\right)^2\ge 1 \quad \hbox{ for } t\le T_{1}.
 \end{align}
\end{proof}

(\textit{II: uniform bounds}). Now that we have a uniform time interval for
which all $u_n^{(N)}$ exist, we look to
establish bounds for  $u_i^{(N)}(t)$ and $iu_i^{(N)}(t)$ which
are also independent of $N$.  Toward this, we define a ``Catalan-like'' sequence
\begin{align}
         A_{2L+\ell} &= 2L+ \ell, 
     \\\frac{A_{n+1}}{n+
        1} &= \frac {2L} { n+1-\ell  -2L}\sum_{k = \ell +1 }^n A_kA_{n-k+\ell+1},
\qquad n > 2L +\ell.
 \label{almostcat}
\end{align}

\begin{theorem}
For $t \in [0, T_1]$, 
\begin{equation}
        u_i^{(N)}(t) \le \frac{A_i}{i} t^{ ((i-\ell)/2L)  -1}, \qquad i \ge
2L+ \ell. \label{indubnd} 
\end{equation}
\end{theorem}

\begin{proof}
The base case holds  since $i = 2L+\ell $  in  (\ref{indubnd}) reduces to
$u_i^{(N)}(t) \le 1$.
To establish the inductive step, we substitute the hypothesis into (\ref{ufineqell})
to give
\begin{align}
        \dot u_{n+1}^{(N)} &\le   \sum_{k = \ell+1}^{n} k(n-k+\ell+1)
\ufin{k}{N} \ufin{n-k+\ell+1}{N}\\
        &\le  \sum_{k = \ell+1}^n A_kA_{n-k+\ell+1} t^{ ((n+1-\ell)/2L) -2}.
\end{align}
Since  $u_{n+1}^{(N)}(0) = 0$ for $n>\ell+L$, integrating the inequality
gives
\begin{align}
        u_{n+1}^{(N)}(t) &\le \frac {1} { (n+1-\ell)/2L  -1}\sum_{k = \ell
+1 }^n A_kA_{n-k+\ell+1} t^{ ((n+1-\ell)/2L)  -1} \\&= \frac {2L} { n+1-\ell-2L}\sum_{k
= \ell +1 }^n A_kA_{n-k+\ell+1} t^{ ((n+1-\ell)/2L)  -1} \\&=
        \frac{A_{n+1}}{n+
        1} t^{ ((n-\ell+1)/2L) -1},
\end{align}
which establishes the induction step.
\end{proof}

Summing over cluster sizes, we now observe that 
\begin{equation}
        \sum_{n =2 L+\ell}^N n\ufint{n}{N} \le \sum_{n =2 L+\ell}^\infty
A_n
t^{ ((n-\ell)/2L)  -1}, \label{sumcbound}
\end{equation}
noting that the right hand side is independent of $N$.
We now show this series converges for some small time.

\begin{theorem}
There is a positive time $T_2 \in (0,T_1]$ such that the sequence $\sum_{n
=2 L+\ell}^\infty A_n
t^{ ((n-\ell)/2L)  -1}$ converges for $t\le T_2$.\end{theorem}

\begin{proof}

 Through induction on $n \ge 2L+\ell$, we will show
\begin{equation}
        A_n \le (2L+\ell)^3\gamma^{((n-\ell)/2L)-1}/n^{2},  \quad n \ge 2L+\ell,
\label{abound}
\end{equation}
where
\begin{equation}
\gamma = 14(2L+\ell+2)(2L+\ell)^3.
\end{equation}
We check that $A_{2L+\ell} = 2L+\ell$
satisfies the inequality by inspection. To show the inductive step, we'll
use an inequality associated with proving a positive radius of convergence
for the Catalan generating function,  shown by Villarino  \cite{villarino2016convergence}:
\begin{equation}
\frac{2}{n^2} + \sum_{k = 1}^{n-1} \frac{1}{k^2(n-k)^2} < \frac{6}{(n+1)^2},
\label{catsum1} \qquad \hbox{ for } n\ge 37.
\end{equation}
From an exhaustive listing of the left hand side of (\ref{catsum1}) for
 $n<37$, we can show that this implies 
\begin{equation}
 \sum_{k = 1}^{n} \frac{1}{k^2(n+1-k)^2}  < \frac{14}{(n+1)^2},
\qquad \hbox{ for } n\ge 2.
\end{equation}For  $n>2L+\ell$, we substitute (\ref{abound})
into (\ref{almostcat}) and note that $(n+1)/(n+1-\ell-2L)\le 2L+\ell+2$
for $n\ge \ell +2L+1$ to obtain 
\begin{align}
        A_{n+1}\le \frac{n+1}{n+1-\ell-2L}\sum_{k =2 L+\ell}^{n} \frac{(2L+\ell)^6\gamma^{((n+1-\ell)/2L)-2}}{k^{2}(n-k+\ell+1)^{2}}
    \\
        \le (2L+\ell+2)(2L+\ell)^6 \gamma^{((n+1-\ell)/2L)-2} \sum_{k =2
L+\ell}^{n}
\frac{1}{k^{2}(n+1-k)^{2}} \\
 <(2L+\ell)^3  [14(2L+\ell+2)(2L+\ell)^3/\gamma]\gamma^{((n+1-\ell)/2L)-1}
/(n+1)^2  \\=    (2L+\ell)^3\frac{\gamma^{((n+1-\ell)/2L)-1}}{(n+1)^2}.
\end{align}
which establishes the inductive step. From the Cauchy Hadamard theorem, it
follows that the power series $\sum_{n=2 L+\ell}^\infty A_n t^{ ((n-\ell)/2L)
 -1}$ converges
for $|t| \le 1/\gamma:= T_2$.   

\end{proof}

(\textit{III: Uniform convergence}). For fixed $n \ge \ell +1$, we can establish
equicontinuity for the function $\{u_{n}^{(N)}(t)\}_{N\ge 2L+\ell}$ for $t
\in [0, T_{2}]$. This holds  by noting that we can bound fractions by $0
\le u_{n}^{(N)}(t)
\le 1$ and  derivatives by  $|\dot u_n^{(N)}(t)| \le 2$. When $n = \ell +1$,
we may apply the Arzela-Ascoli theorem to show that there is a subsequence
$N_{k;\ell+1}$  such that $u_{\ell+1}^{(N)}$ converges uniformly on this
subsequence to some $u_{\ell+1}$ as $k \rightarrow \infty$.  A subsequence
of $N_{k;\ell+2}$ of $N_{k;\ell+1}$ can then be found so that  $u_{\ell+2}^{(N)}$
converges uniformly to some $u_{\ell+2}$.  Continuing in this fashion, for
any $j \ge \ell+2$, we may 
take subsequences $N_{k;j}$ of $N_{k;(j-1)}$ so that $u_{j}^{(N)}$ converges
to $u_j$. Then
along the diagonal subsequence $N_k = N_{k;k}$, all $u_n^{(N)}$ converge
uniformly to $u_n$. 

\begin{theorem}
For $j \ge 0$, $\sum_{n = \ell+1}^{N_{k}} i^j u_i^{N_k}(t)$ converges uniformly
to  $ \sum_{n= \ell+1}^\infty i^{j}u_i(t)$ for $t \in [0, T_2]$. In particular,
 $ \sum_{n = \ell+1}^\infty
u_i(t) = 1 -  t$ and  $ \sum_{n = \ell+1}^\infty
iu_i(t) = m_1(0) - \ell t$.    
\end{theorem}

\begin{proof}
To show moments are finite, we use  Fatou's lemma. From   (\ref{indubnd})
and (\ref{abound}), there is a constant
$C>0$ such that  
\begin{align}
 \sum_{n = \ell+1}^\infty n^j u_n(t) \le \liminf_{N_k\rightarrow \infty}
\sum_{n = \ell+1}^\infty n^j u_n^{(N_k)}(t) \le  C\sum_{n = \ell+1}^\infty
n^j\frac{(\gamma t)^{ (n-\ell+1)/2L -1}}{n+
        1}. \label{fatou}
\end{align}
This sum converges for $t \in [0,T_2]$. For any $k \ge \ell +1$ and $N'<N_k$
we can then bound 
 \begin{align}
        \left|\sum_{i = \ell+1}^{N_k} i^{j}u_i^{(N_k)} - \sum_{i = \ell+1}^\infty
i^{j}u_i\right|
\le \left|\sum_{i = \ell+1}^{N'} (i^{j}u_i^{(N_{k})} - i^{j}u_i)\right|+
\left|\sum_{i
= N'+1}^{N_k} i^{j}u_i^{(N_{k})}\right|+ \left|\sum_{i = N'+1}^\infty
i^{j}u_i \right|.
\end{align}

Let $\varepsilon >0$.   Because of the uniform bound (\ref{sumcbound}) and
(\ref{fatou}), we
can choose a value $N'$ large enough so that
the
second and third terms are each less than $\varepsilon/3$. We can then let
$N_{k}$ be sufficiently
large so that the first term is also bounded by $\varepsilon/3$. This proves
our claim of uniform convergence of moments.  The particular cases of the
zeroth and
first moments follow from  (\ref{finmom}).

\end{proof}

Taking uniform limits of  the right hand side of (\ref{ufineqell}), we find
that $\dot u_i^{(N)}$ must converge uniformly to the right-hand side of (\ref{realmaineqn}),
and that $\dot u_i^{(N)}\rightrightarrows
\dot u_i$. 
This establishes existence of pre-gelation solutions of (\ref{realmaineqn})
for $t \in [0,T_2]$.

(\textit{IV: uniqueness}). To show uniqueness, we  note that the solution
we have found satisfies $m_1(t)
= m(0) - \ell t$.  The equation for  $u_{\ell+1}$ is
closed, and can easily be solved to show
\begin{equation}
u_{\ell+1}(t) = u_{\ell+1}(0) \left(\frac{m_1(t)}{m_1(0)}\right)^{2n/\ell}.
\label{firstsol}
\end{equation}
For $n \ge \ell+1$, we can  use  
$q$-coordinates (\ref{qness}), and  solve $q_n$ in terms of $q_1, \dots,
q_{n-1}$.  The integration factor
can then be expressed in terms of the total mass as
\begin{equation}
        \exp\left(\int \frac{2n-\ell}{m_{1}(t)} dt \right) = m_{1}(t)^{\frac{\ell
-
2n}{\ell}},  
\end{equation} 
and so we obtain
\begin{equation}
u_n(t) = m_1(t)^{\frac{2n}{\ell}}\left(\int_0^t m_{1}(s)^{\frac{-2n}{\ell}-2}
\sum_{i= \ell+1}^{n-1} i(n+ \ell-i)u_j(s) u_{n-j+1}(s)ds+u_n(0)m_1(0)^{-\frac{2n}{\ell}}\right).
\label{mainiter}
\end{equation}

This gives a unique recursive solution formula for $u_n$. Furthermore, 
since $u_{\ell+1}(t)$ is smooth, we can deduce from (\ref{mainiter}) that
each $u_n$ for $n>\ell+1$
is also smooth. The proof for Theorem \ref{largeeuthm} is complete.

\section{Properties of large-cluster systems} \label{sec:mom}

\subsection{Moments}

Having established the uniqueness of pre-gelation solutions, and also that
all moments exist for these solutions, we now give explicit formulas for
computing moments.  

\begin{theorem} \label{momentthm}
For $j \ge 1$, moments for pre-gelation solutions satisfy
\begin{align}
\dot m_k &=\frac1{m_1^2}   \sum_{n,j \ge \ell +1} [(n+j-\ell)^k -2n^{k} ]nju_ju_n
\\ &=- 2\frac{m_{k+1}}{m_1}+ 
\sum_{q+r+s = k}
  {k \choose q,r,s}  \frac{m_{r+1}m_{s+1}}{m_1^2}(-\ell)^{q}. 
\end{align}
\end{theorem}
\begin{proof}

This follows from a direct calculation on (\ref{realmaineqn}), where we multiply
$\dot u_n$
by $n^k$:
\begin{align}
\dot m_k &= \sum_{n \ge \ell +1} n^k \dot u_n =   \sum_{n \ge \ell +1}n^k\left[\sum_{i
= \ell+1} ^{n-1} \frac{i(n+\ell-i)}{m_1^2}u_i u_{n+\ell
- i} - \frac{2nu_n}{m_1^{2}} \right] \\
 &=   \sum_{n \ge \ell +1}\left[\sum_{j
\ge \ell+1}  \frac{(n+j-\ell)^knj}{m_1^2}u_ju_{n} -n^{k} \sum_{j \ge \ell
+1}\frac{2nju_ju_n}{m_1^2}
\right] \\
&=  \frac 1{m_1^2}   \sum_{n,j \ge \ell +1} [(n+j-\ell)^k -2n^{k} ]nju_ju_n.
\end{align}
The factor $(n+j-\ell)^k$ can then be expanded from the multinomial theorem
 to obtain
\begin{align} 
\dot m_k &=  \frac 1{m_1^2}   \sum_{n,j \ge \ell +1}\left(\left[ \sum_{q+r+s
= k}   {k \choose q,r,s} n^{r+1}j^{s+1}(-\ell)^{q}\right]
-2 jn^{k+1}\right)u_ju_n \\
&=- 2\frac{m_{k+1}}{m_1}+ 
\sum_{q+r+s = k}
  {k \choose q,r,s}  \frac{m_{r+1}m_{s+1}}{m_1^2}(-\ell)^{q}.
\end{align}

\end{proof}

Theorem \ref{momentthm} provides a hierarchy of first-order linear equations
that can be used to solve for arbitrary moments. For an example, we give
formulas for the second and third moments under monodisperse $k$-mer
initial conditions 
\begin{equation}
u_n(0) = \delta_{n,k} = \begin{cases}1 & n = k, \\
0 & n \neq k, \\
\end{cases} \label{kmer}
\end{equation}
where $k> \ell$.  The initial $j$th moment is then $m_j(0) = k^j$ and from
(\ref{m1loss}), $m_1(t) = k-\ell t$.
For $m_2$,
we may write 
\begin{equation}
\dot m_2 = \frac{1}{m_1^2}(2m_2^{2}-
4\ell m_2m_1+\ell^2m_{1}^2). \label{m2diff}
\end{equation}
From elementary methods, we find 
 \begin{equation}
m_2(t)  = (k-\ell t)\frac{k^2-2k\ell t+ \ell^2
t}{k -2 k t + \ell t}. \label{m2}
\end{equation}

Calculations for $m_3$, while more cumbersome, are still straightforward.
 We will outline the steps involved in solving for $m_3$.   The equation
for $m_3$ can be written as
\begin{align}
        \dot m_3+ m_3\left(\frac{6\ell}{m_1} - 6\frac{m_2}{m_1^2}\right)
= -6\ell\left(\frac{m_2}{m_1}\right)^2+6\ell^2\left(\frac{m_2}{m_1}\right)-\ell^3.
\end{align}
The integration factor can be expressed as a rational function, with
\begin{align}
        I(t) &= \exp\left(6 \int \frac{\ell}{k-\ell t} - \frac{1}{k-\ell
t}
\frac{k^2-2k\ell t+\ell^2 t}{k-2kt+\ell t}
dt\right) \\
&= (k-\ell t)^{-3}(k-2kt+\ell t)^{3}. 
\end{align}
We may then write an explicit formula for $m_3$ as 
\begin{align}
        &m_3(t) =I(t)^{-1}\left[k^3+ \int_0^t I(s)\left(-6\ell\left(\frac{m_2}{m_1}\right)^2+6\ell^2\left(\frac{m_2}{m_1}\right)-\ell^3\right)ds
\right].
\end{align}
Through standard integration methods, we find
\begin{align}
m_{3}(t)=\frac{k-\ell t}{(k-2kt+\ell t)^3}(\ell^2(\ell - 2k)^3 t^3 + k\ell(2\ell^3
- 6k\ell^2 + k^2\ell + 6k^3) t^2
- k^2\ell(\ell - 2k)(\ell - 4k) t + k^5). \label{m3}
\end{align}

A generating function can also be employed to find moments. If we let $C(z,t)
= \sum_{n = \ell+1}^\infty e^{-nz} c_n(t)$, from a reindexing similar to
Theorem \ref{momentloss} we obtain the partial differential
equation \begin{equation}
\partial_t C = \frac{e^{\ell z}}{m^2_1(t)} (\partial_z C)^2 +\frac{2}{m_{1}(t)}\partial_z
C. \label{genfun}
\end{equation}
By taking derivatives with respect to $z$ and evaluating at $z = 0$, we recover
differential equations for moments.  From the decay bounds  (\ref{abound})
for cluster fractions, all partial derivatives with respect to both $t$ and
$z$ are well defined for $C(z,t)$ for $t \in [0,T_2]$ and for sufficiently
small $z$. Differentiating (\ref{genfun}) with respect to $z$ gives
\begin{equation}
\partial_t C_{z} = \frac{e^{\ell z}}{m^2_1(t)} (2C_zC_{zz}+\ell(C_z)^2)+
\frac{2}{m_{1}(t)}C_{zz}.
\end{equation}
The equation evaluated at $z = 0$ reduces to $\dot m_1 = -\ell$.  Differentiating
again gives
\begin{equation}
\partial_t C_{zz} = \frac{e^{\ell z}}{m^2_1(t)} (2(C_{zz})^2+4\ell C_zC_{zz}+2C_zC_{zzz}+\ell^2(C_{z})^2)+
\frac{2}{m_{1}(t)}
C_{zzz}.
\end{equation}
Now plugging in $z = 0$ gives (\ref{m2diff}). Repeated differentiation of
(\ref{genfun}) can then be used to find equations for higher moments.

\subsection{Polynomial solution form for monomer loss} \label{sec:poly}

For pure dimer initial conditions $u_k(0) = \delta_{2,k}$ and $\ell = 1$,
it follows that
$m_1(t) = 2-t$  and 
from (\ref{firstsol}), 
\begin{equation}
u_2(t) = \frac 1{16}m_{1}(t)^4. \label{u2soln}
\end{equation}
To find $u_3$, we can use (\ref{mainiter}) to obtain
\begin{align}
u_3(t) = m(t)^6 \left(\int_0^t 4m^{-8} u_2^2ds\right) 
 = \frac{m_{1}^6}{2^5}-\frac{m_{1}^7}{2^6} . \label{u3soln}
\end{align}
We can continue using (\ref{mainiter}) to solve for fractions of larger clusters,
and observe that solutions appear to be polynomials in $m_1$.  In fact, polynomial
solutions arise  for all cases when $\ell = 1$ with no initial monomers and
a positive fraction of dimers.

\begin{theorem} \label{polythm}

Let $u_n$ be the pre-gelation large-cluster solution for $\ell = 1$ with
bounded initial conditions satisfying $u_2(0) >0$. For $n\ge 2$, cluster
fractions for large-cluster solutions take the form 
\begin{equation}
        u_n = \sum_{j =2n}^{3n-2}a_{j}^{(n)}m^j_1
 \label{seriesform}
\end{equation}
for real constants  $a_{j}^{(n)}$.  Furthermore, $a_{2n}^{(n)}$ and $(-1)^na_{3n-2}^{(n)}$
are both positive.
\end{theorem}

\begin{proof}
Let $\mathcal P$ be the space of polynomials in $m_1(s) = m_{1}(0)-s$ (equivalent
to the space of polynomials in $s$). From the solution formula (\ref{mainiter}),
for $n\ge 3$ we can obtain
$u_{n}$ from $(u_2, \dots,
u_{n-1})$ through a composition of two maps:

\begin{enumerate}
\item $T_1^{(n)}: \mathcal P ^{n-2} \rightarrow \mathcal P$ is defined as
        \begin{equation}
        (f_2, \dots, f_{n-1})  \mapsto 
 \sum_{j= 2}^{n-1} j(n-j+1)f_j
f_{n-j+1}. \label{t1}
        \end{equation}
        This map is  the quantity in the integral of (\ref{mainiter}) before
multiplication by $m_1^{2-n}$. Coefficients can be found through standard
polynomial
multiplication, and so the range is also a polynomial in $m_1$. 
\item The remaining operations in (\ref{mainiter}) are given by $T_2^{(n)}$,
which is the composition $T_2^{(n)}:T_{2;1}^{(n)}
\circ T_{2;2}^{(n)}
\circ T_{2;3}^{(n)}
$ of linear maps. It is sufficient to describe each of their actions on the
monomial $m_1^p$ and then extend to $f \in \mathcal P$ by linearity: 
\begin{itemize}
\item $T_{2;1}^{(n)}$  is a multiplication operator defined by 
\begin{equation}
T_{2;1}^{(n)}[m^{p}_1] = m_1^{p-2n-2} .
\end{equation}  
\item $T_{2;2}^{(n)}$ is an integration operator defined  by 
\begin{equation}
 T_{2;2}^{(n)}[m^{p}_1] =   \frac{1}{p+1}(m_1(0)^{p+1}- m^{p+1}_1 )  \quad
\hbox{ for }  p \neq -1.
 \end{equation}
 \item Finally, $T_{2;3}^{(n)}$ is  defined by 
\begin{equation}
T_{2;3}^{(n)}[m_{1}^p] =   m_1^{2n}(m_{1}^p+u_n(0)m_1(0)^{-2n}). 
\end{equation}
\end{itemize}
\end{enumerate}
The composition of these maps is then \begin{equation}
T_2^{(n)}[m^p_1] =\left(u_n(0)m_{1}(0)^{-2n}+\frac{m_{1}(0)^{p-2n-1}}{p-2n-1}
\right)m^{2n}_1-\frac{1}{p-2n-1} m_1^{p-1}, \qquad \hbox{for } p \neq 2n+1.
\label{t2full}  \end{equation}

We now show that  $T^{(n)} = T_2^{(n)} \circ T_1^{(n)}: \mathcal P \rightarrow
\mathcal P$, meaning that $u_n \in \mathcal P$ for $n \ge 2$. This is achieved
by  induction.  Our claim that is that $u_n \in \mathcal P$ and that the
lowest power of $m_1^p$ in $u_n$ is $\beta_n = 2n$. From (\ref{firstsol}),
this is certainly true for $u_2$.  
Under the induction hypothesis, the lowest exponent of $T_1^{(n)}[u_n]$ is
\begin{equation}
\min_{2\le j\le n-1} \{\beta_{n-j+1}+\beta_j
 \} = 2n+2.
\end{equation}
Since this power is not $2n+1$, the map (\ref{t2full}) for $T^{(n)}_2$ is
well-defined, and so $T^{(n)} : \mathcal P \rightarrow \mathcal P$. 

The terms found from the $T_2$ are $2n$ and powers produced in $T_1$ reduced
by one.  It then follows that
\begin{align}
        \beta_n &  =   \min\{\min_{2\le j\le n-1} \{\beta_{n-j+1}+\beta_j-1
 \}, 2n\}  =   \min\{2n+2, 2n\}   = 2n.
\end{align}
The coefficient in front of the $m_{2n}$ is positive for all values of $n$,
and so it follows that $a_{2n}^{(n)}>0$.

 Denoting the largest exponent of $m_1^p$ in $u_n$ as  $\alpha_n$, we observe
from (\ref{u2soln}) and (\ref{u3soln}) that $\alpha_2 = 4$ and $\alpha_3
= 7$. Another simple induction argument shows
\begin{align}
        \alpha_n &=     \max\{ \max_{2\le j\le n-1} \{\alpha_{n-j+1}+\alpha_j-1
 \} , 2n\} = 3n-2. \label{alpha}
\end{align}
From an induction hypothesis that $a_{3i-2}^{(i)}(-1)^i>0$ for $2 \le i \le
n-1$, for $n\ge 3$ we see that the coefficient in front of $m^{3n-2}$ has
sign\begin{equation}
 \mathrm{sgn}(a_{3n-2}^{(n)}) =  \mathrm{sgn} \Big[ -\frac{1}{n-2} \sum_{i
= 2}^{n-1} i(n-i+1) a^{(i)}_{3i-2}a^{(n-i+1)}_{3(n-i+1)-2} \Big] = \mathrm{sgn}(-(-1)^{i}(-1)^{n-i+1})
= (-1)^n. \end{equation} 
\end{proof}

If we allow for a positive fraction of monomers, we should not expect polynomial
solutions in general. For a counterexample, we will still consider an emission
size of  $\ell = 1$, but now we impose initial conditions $u_1(0) = u_2(0)
= 1/2$,
and $u_n(0) = 0$ for $n \ge 3$.  Then $m_1(t) = 3/2-t$.   The closed equation
for monomers is
\begin{equation}
\dot u_1 = \frac{u_1^2}{m_1^2}-2\frac{u_1}{m_1}.
\end{equation}
In $q$-coordinates, the equation is separable, with\begin{equation}
m_1\dot q_1 = q_{1}^2-q_1 
\end{equation}
and solution\begin{equation}
u_1 =   \frac{m_1^2}{(m_{1}+3)}.
\end{equation}
The equation for $u_2$ is
\begin{equation}
\dot u_2 = \frac{4u_1u_2}{m_1^2} - 4\frac{u_2}{m_1}.
\end{equation}
Substituting the solution for $u_1$, the equation is solvable without appealing
to $q$-coordinates, with  
\begin{equation}
\frac{du_2}{u_2} = 4 \left(\frac 1{m_1+3}- \frac 1{m_1} \right)dt \quad \Rightarrow
\quad  u_2 = \frac{3^4}{2}\left(\frac{m_1}{m_1+3}\right)^4. 
\end{equation}

Finally, the equation for  $u_3$ is
\begin{equation}
\dot u_3 = \frac{4 u_2^2+ 6u_1u_3}{m_1^2}- \frac{6u_3}{m_1}.
\end{equation}
We use $q$-coordinates again to obtain\begin{equation}
m_1\dot q_3 = -5 q_3+4q_{2}^2+6q_1q_3.
\end{equation}
We solve through an integration factor to obtain\begin{equation}
q_3(t) = \frac{1}{I(t)} \left(\int_0^t \frac{4 q_2^2}{m_1} I(s)ds\right),
\label{u3ref}
\end{equation}
where\begin{equation}
I(t) = \exp\left(\int \frac{5-6q_1}{m_1} dt\right) = \frac{(m_{1}+3)^6}{m_{1}^5}.
\end{equation}
Trimers are then found to have a fraction
of\begin{equation}
u_3(t) =\frac{3^{8} m_1^6(3-2m_1)}{9(m_1+3)^7}.
\end{equation}
Solutions appear, then, to be rational functions of $m_1$, although we have
not been able to prove anything rigorous in this direction.
We also note that we should not expect polynomial solutions if we increase
the emission size, even for simple initial conditions.  For a  counterexample,
we can take $\ell = 3$ with $u_n(0) = \delta_{n,4}$. From (\ref{firstsol}),
size-4 clusters have fraction $u_4 = (m_1/4)^{8/3}$.

\section{Simulations and gelation behavior} \label{sec:sims}

\begin{figure}[htbp]
    \centering
    \begin{subfigure}[b]{0.45\textwidth}
        \includegraphics[width=\textwidth]{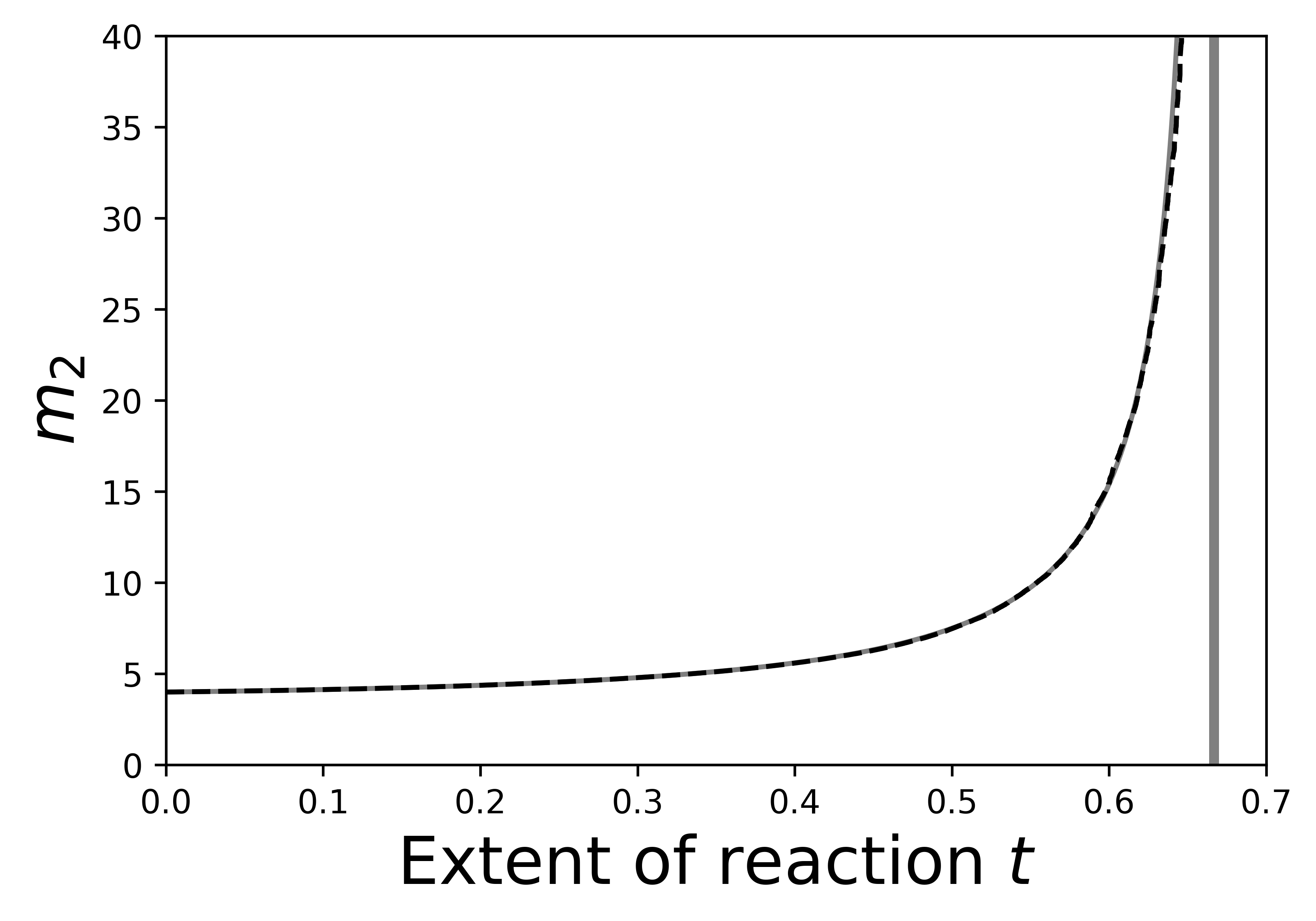}
    \end{subfigure}
    \hfill
    \begin{subfigure}[b]{0.45\textwidth}
        \includegraphics[width=\textwidth]{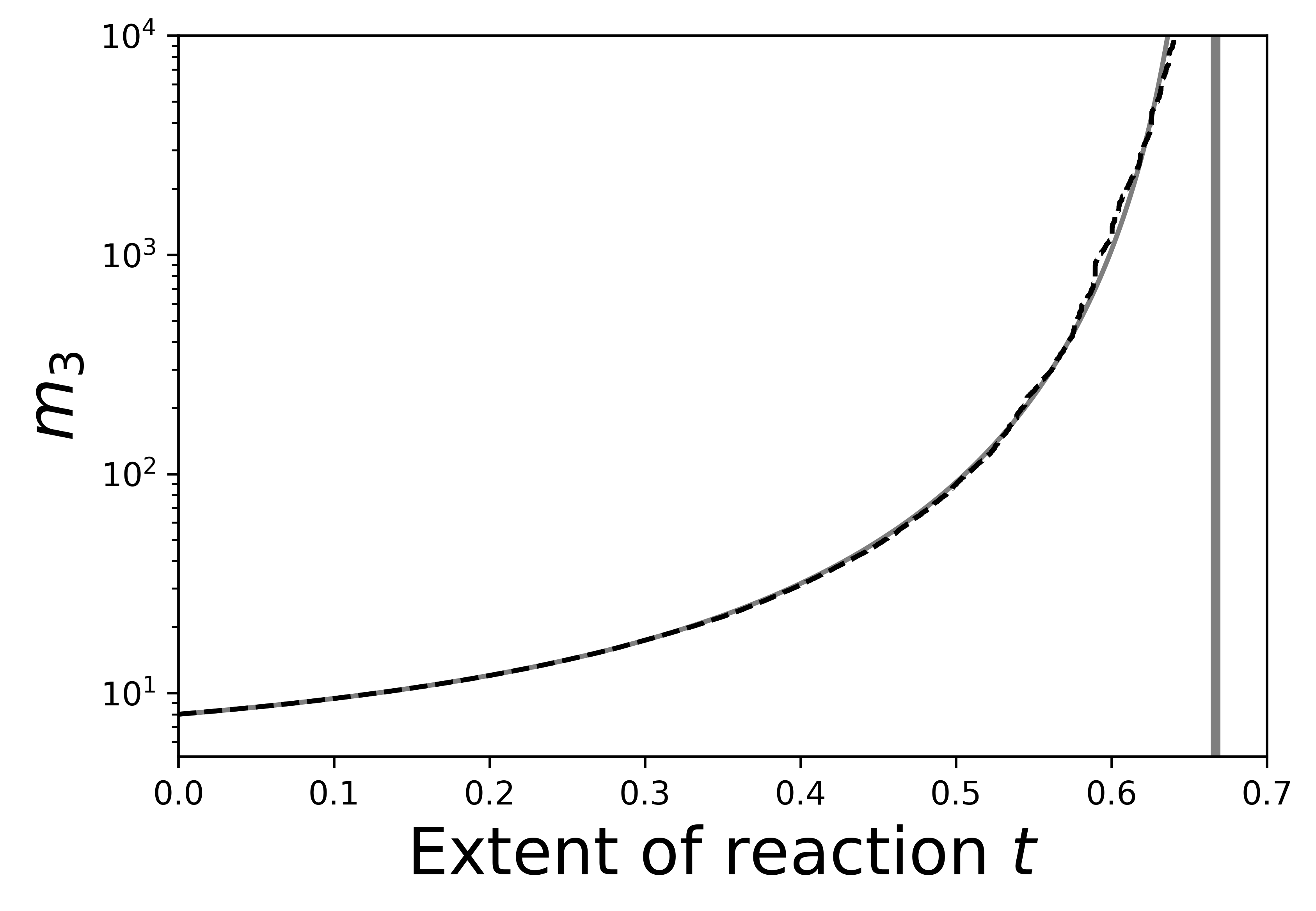}
    \end{subfigure}
    \caption{\textbf{Second and third moments}.
Limiting formulas of $m_i(t) $ (transparent solid lines) and prelimit moments
$m^{(N)}_i(t)$  with $N = 10^6$ (dashed lines) for $i = 2$
(left) and $i = 3$ (right, with log scaling in the $m_3$-axis). The grey
thick vertical line is at $t_{\mathrm{gel}}
= 2/3$.}\label{momentpic} 
\end{figure}

In this section, we present results from simulations of the Markov process
$u^{(N)}(t)$ for $N = 10^6$ clusters, with pure dimer initial conditions
$u^{(N)}_n(0) = N\delta_{n,2}$, and $\ell = 1$.  The cluster fractions $u^{(N)}(t)$
 are compared to limiting cluster fraction formulas $u_n$ generated by (\ref{mainiter})
with initial conditions $u_n = \delta_{2,n}$. The maps for $T_1^{(n)}$ and
$T_2^{(n)}$ described in (\ref{t1}) and (\ref{t2full})
for Theorem (\ref{polythm}) provide an iterative algorithm for generating
coefficients
in (\ref{seriesform}).  Formulas for 
$u_2$ and $u_3$ are given in (\ref{u2soln}) and (\ref{u3soln}).  With coefficients
rounded to the  nearest millionth, we generate \begin{align*}
u_4 &= 0.005859t^{10} - 0.023438t^{9} + 0.023438t^{8}, \\
u_5 &=-0.002686t^{13} + 0.016113t^{12} - 0.032227t^{11} + 0.021484t^{10},
\\
u_6 &=0.001389t^{16} - 0.011108t^{15} + 0.033325t^{14} - 0.044434 + 0.022217t^{12},
 \\
u_7 &= -0.000778t^{19} + 0.007782t^{18} - 0.031128t^{17} + 0.062256t^{16}
- 0.062256t^{15} + 0.024902t^{14}, \\
u_8 &= 0.000462t^{22} - 0.005545t^{21} + 0.027723t^{20} - 0.073929t^{19}
+ 0.110893t^{18} - 0.088715t^{17} + 0.029572t^{16},  \\
u_9 &=  -0.000287t^{25} + 0.004011t^{24} - 0.024068t^{23} + 0.080225t^{22}
- 0.160451t^{21} + 0.192541t^{20}\\ &- 0.128361t^{19} + 0.036674t^{18}, \\
u_{10} &=  0.000184t^{28} - 0.002941t^{27} + 0.020584t^{26} - 0.082337t^{25}
+ 0.205842t^{24} - 0.329347t^{23}\\ &+ 0.329347t^{22} - 0.188198t^{21} +
0.047050t^{20}.
\end{align*}

The corresponding prelimit moments $m^{(N)}_2$ and  $m^{(N)}_3$ are compared
to limiting moments 
\begin{equation}
m_2(t) = \frac{(2-t)(4-3t)}{2-3t}, \qquad m_3(t) = \frac{(2-t)(-27t^3+84t^2-84t+32)}{(2-3t)^3},
\label{m23}
\end{equation}
where we have substituted $k = 2$ and $\ell = 1$ in formula
(\ref{m2}),(\ref{m3}).
We stress that we
have only established short time existence and uniqueness in  Theorem \ref{largeeuthm}.
 However, we show that limiting quantities appear to agree with numerical
simulations until a (conjectured) gelation time of $t_{\mathrm{gel}} = 2/3$
where  $m_2(t_{\mathrm{gel}}^-) = \infty$ in (\ref{m23}).

\begin{figure}
\includegraphics[width=.5\textwidth]{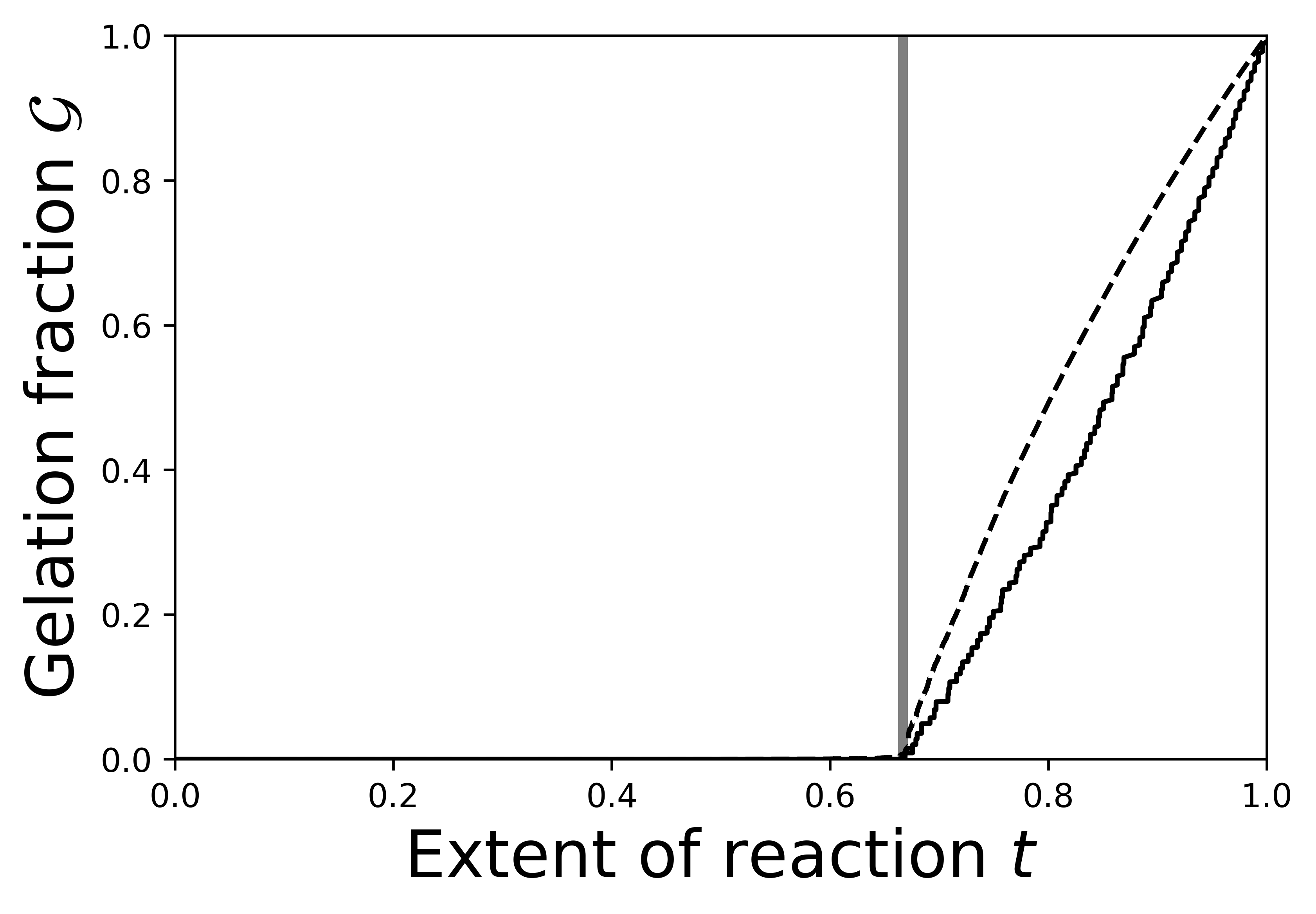}
\caption{\textbf{Gelation fractions}.  Gelation fractions $\mathcal G^{N}(t)$
for the Ziff-Stell 
(dashed line) and Stockmayer (solid line) models. The grey thick vertical
line is at $t_{\mathrm{gel}}
= 2/3$.} \label{gelpic}
\end{figure}

Under general initial conditions,  the iterative equation (\ref{mainiter})
 used for generating explicit formulas for $u_n$ holds when $m_1(t) = m_{1}(0)-\ell
t$. We showed in Theorem \ref{momentloss} that this is true when $m_2<\infty$.
 We should, however, expect $\dot m <-\ell$ past the gelation time $t_{\mathrm{gel}}
= \min \{t \in [0,t_{\mathrm{ex}}):m_2(t^{-}) = \infty\}$, if such a time
exists. Formula (\ref{m2}) suggests that  under monodisperse $k$-mer initial
conditions (\ref{kmer}),   $t_{\mathrm{gel}} = k/(2k-\ell)$. The
exhaustion time, in which all particles are removed from the
system, is $t_{\mathrm{ex}} = k/\ell$.  
If $k>\ell$, then $t_{\mathrm{gel}}<t_{\mathrm{ex}}$ and so we should expect
gelation to occur before the system exhausts itself. In Fig. \ref{momentpic}
we plot $m^{(N)}_2$ and  $m^{(N)}_3$ against $m_2$ and  $m_3$. We find good
agreement for both moments until values near $t_{\mathrm{gel}} = 2/3$.

For an estimate for the relative mass of the gel compared to total mass after
$t_{\mathrm{gel}}$, we consider the numerical gel fraction $\mathcal G^{(N)}(t)$
of particles belonging
to the largest cluster, or 
\begin{equation}
\mathcal G^{(N)}(t) = \frac{\mathrm{argmax}\{i:U_i^{(N)}(t) \neq 0\}}{\sum_n
n U_n^{(N)}(t)}. \label{gelfrac}
\end{equation}
Several models exist for how the gel interacts with finite sized clusters
(the sol).  The simplest model for finding the gelation fraction  is similar
to a model of Ziff and Stell \cite{ziff1980kinetics}, where the gel is permitted
to interact with the sol.  In terms of the sampling procedure, we can obtain
$\mathcal G^{(N)}(t)$ through simulating (\ref{cadlag}) and recording the
largest cluster after each time step with no further intervention.     Since
we assume sampling without replacement in the Markov model, the gel cannot
be selected twice to react with itself, an interaction found the popular
cross-linking polymer models of Flory \cite{flory1953principles}.   We can
also consider the model of Stockmayer  \cite{stockmayer1943theory},
in which only the sol is only  allowed to interact with itself. For simulating
this model, we restrict clusters to interact when they exceed a certain threshold
size.   See Fig. \ref{gelpic} for the gelation fraction $\mathcal G^{(N)}(t)$
under both the Ziff-Stell and Stockmayer models.  In the Stockmayer model,
we used a minimal size threshold of $N/100 = 10^4$ for clusters which are
not allowed to interact.  For both models, we find that a `knee' occurs 
approximately at $t_{\mathrm{gel}} = 2/3$.  

In Fig.\ \ref{fracpic} we plot cluster fractions for $u_i$ for $i = 2, \dots,
6$. We find good agreement between the cluster fraction formulas and  the
 Markov process (\ref{cadlag}) until $t_{\mathrm{gel}} = 2/3$.  When $t>
t_{\mathrm{gel}} $,  we find that $u^{(N)}_n(t) < u_n(t)$.  This is due to
the contribution of  particles toward  the gel, which is not reflected in
the  pre-gelation formula (\ref{mainiter}) used for calculating $u_n$.

\begin{figure}[htbp]
    \centering
    \begin{subfigure}[b]{0.45\textwidth}
        \includegraphics[width=\textwidth]{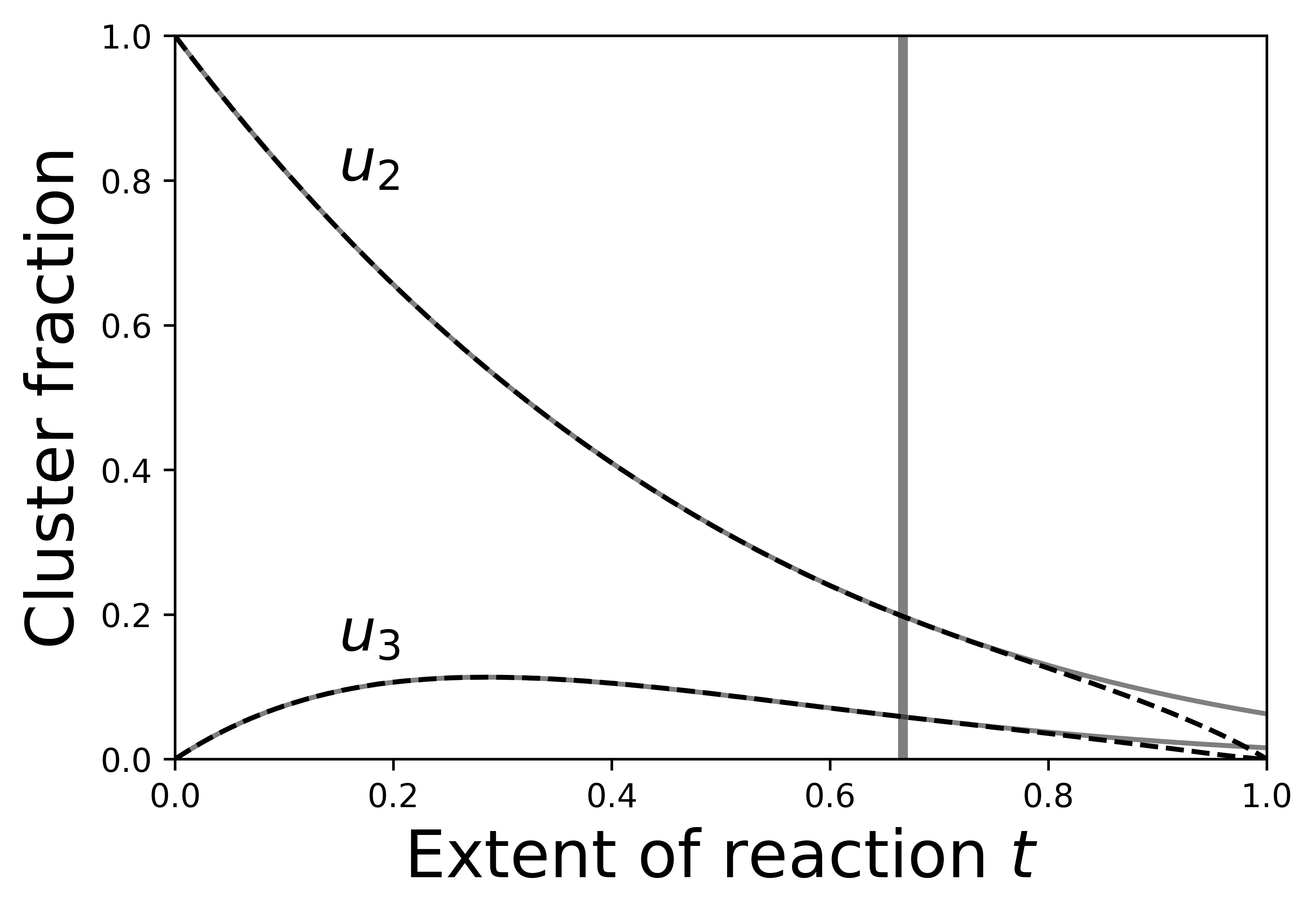}
    \end{subfigure}
    \hfill
    \begin{subfigure}[b]{0.45\textwidth}
        \includegraphics[width=\textwidth]{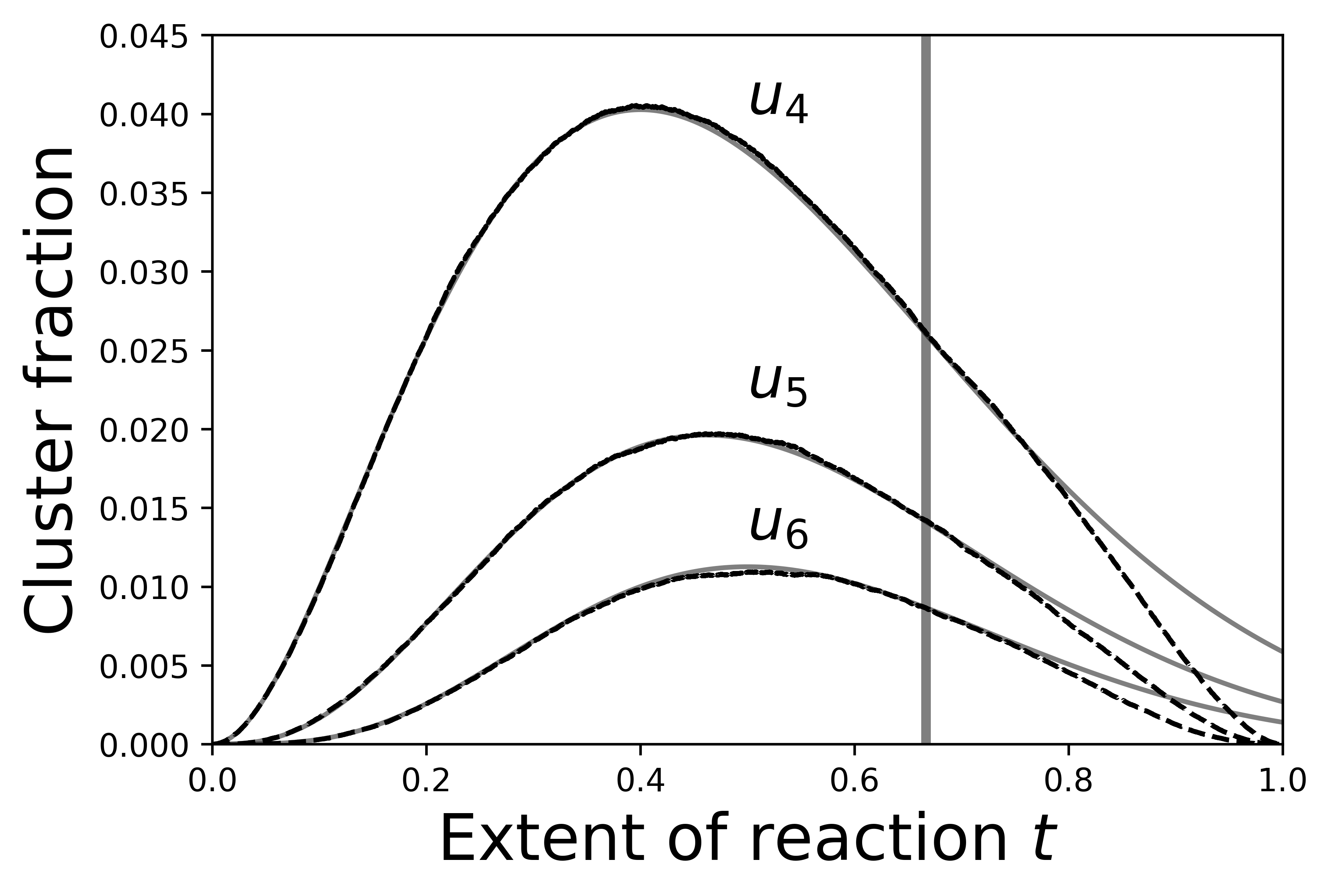}
    \end{subfigure}
    \caption{\textbf{Cluster fractions}.
Limiting formulas $u_n$ (transparent solid lines) and 
$u_n^{(N)}$  for $N = 10^6$ (dashed lines) for $n = 2,3$
(left) and $n = 4,5,6$ (right). The grey thick vertical line is at $t_{\mathrm{gel}}
= 2/3$.}\label{fracpic} 
\end{figure}

\section{Conclusion} \label{sec:disc}

We have studied a model of coagulating particles with mass loss through particle
emission.
Small-cluster solutions were shown to exist on a maximal interval of existence,
in the sense that at exhaustion time particles are no longer available for
reaction (\ref{mainreact}) to occur.   While large-cluster solutions were
only established for a short time, numerical evidence suggests that the iterative
formula for cluster fractions and the moment hierarchy  extends until the
gelation of the system. For both small and large-cluster solutions, we have
relied on the concept of reaction classes which provide us a simple link
between small time cluster growth rates and the required number of reactions
needed to form a cluster.   

Existence methods have, to a large extent,  relied on an analysis of the
generating function for cluster sizes, given as a solution of  a nonlinear
partial differential equation. The generating function is often simple enough
to  solve along characteristics, and in certain cases Lagrange inversion
can  be employed to find explicit formulas for cluster frequencies \cite{wattis2004coagulation,menon2004approach}.
For our model, we found that finding an explicit solution of the generating
function   (\ref{genfun}) was intractable, but we hope in future work to
use singularity analysis  to find asymptotic behavior of cluster fractions
in the post-gelation regime. As demonstrated by Ziff and
Stell  for polymer models \cite{ziff1980kinetics}, the loss  term in (\ref{themainode}),
and subsequently its corresponding generating function, can change depending
  on if and how the gel interacts with the sol.    

  \vspace{10pt}

\textbf{Funding}: The work of all authors is partially supported  by the National Science Foundation under Grant No. 2316289. 

\textbf{Conflicts of interest/Competing interests}:  The authors have no competing interests to declare that are relevant to the content of this article.

\bibliography{coag_7_27_25}

\end{document}